\documentclass{compositio}

\usepackage[utf8]{inputenc}
\usepackage[english]{babel}
\usepackage{newlfont}
\usepackage{amsmath,amssymb,amsfonts,amsbsy,bbm}
\usepackage{mathtools,braket,mathrsfs}
\usepackage[all]{xy}
\usepackage{tikz}
\usetikzlibrary{arrows}
\usepackage[cal=euler]{mathalpha}
\usepackage{stmaryrd}

\usepackage[new]{old-arrows}
\usepackage{extarrows}
\usepackage{textcomp}

\newcommand{\bb}{\mathbb}
\newcommand{\mbf}{\mathbf}
\newcommand{\scr}{\mathscr}
\newcommand{\mrm}{\mathrm}
\newcommand{\et}{{\acute{\mathrm{e}}\mathrm{t}}}
\newcommand{\bs}{\boldsymbol}

\newcommand{\Z}{\ensuremath{\mathbb{Z}}}
\newcommand{\Q}{\ensuremath{\mathbb{Q}}}
\newcommand{\R}{\ensuremath{\mathbb{R}}}
\newcommand{\C}{\ensuremath{\mathbb{C}}}
\newcommand{\A}{\ensuremath{\mathbb{A}}}

\newcommand{\p}{\ensuremath{\mathfrak{p}}}
\newcommand{\q}{\ensuremath{\mathfrak{q}}}

\DeclareMathOperator{\Hom}{Hom}
\DeclareMathOperator{\End}{End}
\DeclareMathOperator{\Gal}{Gal}
\DeclareMathOperator{\St}{St}

\DeclareMathOperator{\Ev}{Ev}
\renewcommand{\det}{\operatorname{det}}
\DeclareMathOperator{\ord}{ord}
\DeclareMathOperator{\rec}{rec}

\newcommand{\too}{\longrightarrow}								
\newcommand{\mapstoo}{\longmapsto}

\newcommand{\into}{\hookrightarrow}
\newcommand{\onto}{\twoheadrightarrow}

\theoremstyle{plain}
\newtheorem{theorem}{Theorem}[section]
\newtheorem{lemma}[theorem]{Lemma}
\newtheorem{proposition}[theorem]{Proposition}
\newtheorem{corollary}[theorem]{Corollary}
\newtheorem{conjecture}[theorem]{Conjecture}
\newtheorem{assumptions}[theorem]{Assumptions}

\renewcommand{\thethmx}{\Alph{thmx}}

\theoremstyle{definition}
\newtheorem{remark}[theorem]{Remark}
\newtheorem{example}[theorem]{Example}
\newtheorem{definition}[theorem]{Definition}

\DeclareFontFamily{U}{wncy}{}
    \DeclareFontShape{U}{wncy}{m}{n}{<->wncyr10}{}
    \DeclareSymbolFont{mcy}{U}{wncy}{m}{n}
    \DeclareMathSymbol{\Sh}{\mathord}{mcy}{"58}

\newcommand{\plectic}[0]{\mbox{\textmarried}}

\def\Xint#1{\mathchoice
{\XXint\displaystyle\textstyle{#1}}%
{\XXint\textstyle\scriptstyle{#1}}%
{\XXint\scriptstyle\scriptscriptstyle{#1}}%
{\XXint\scriptscriptstyle\scriptscriptstyle{#1}}%
\!\int}
\def\XXint#1#2#3{{\setbox0=\hbox{$#1{#2#3}{\int}$ }
\vcenter{\hbox{$#2#3$ }}\kern-.585\wd0}}
\def\mint{\Xint\times}

\makeatletter

\newcommand{\exterior}[1]{\mathop{\mathpalette\exterior@{#1}}}
\newcommand{\exterior@}[2]{%
  \raisebox{\depth}{%
  \fontsize{\sf@size}{0}%
  \m@th
  $\ifx#1\displaystyle\textstyle\else#1\fi\bigwedge$}%
  ^{\mspace{-2mu}#2}
  \kern-\scriptspace
}

\newcommand*{\bfcdot}{\scalebox{0.6}{$\bullet$}}
\renewcommand{\labelitemi}{$\bfcdot$}
\begin{document}

\title{Iwasawa theory and mock plectic points}

\author{Michele Fornea}
\email{mfornea.research@gmail.com}
\address{CRM, Barcelona, Spain.}

\author{Lennart Gehrmann}
\email{gehrmann.math@gmail.com}
\address{Universität Bielefeld, Bielefeld, Germany.}

\classification{11F41, 11F67, 11G05, 11G40}

\begin{abstract}
We use Iwasawa theory, at a prime $p$ inert in a quadratic imaginary field $K$, to study the arithmetic properties of mock plectic invariants for elliptic curves of rank two.
More precisely, under some minor technical assumptions, we prove that the non-vanishing of the mock plectic invariant $\cal{Q}_K$ attached to an elliptic curve $E_{/\Q}$ of even analytic rank $r_\mrm{an}(E/K)\ge2$, and with multiplicative reduction at $p$, implies that the $p$-Selmer rank $r_p(E/K)$ equals $2$.
The proof rests on one inclusion of Perrin-Riou's Heegner point main conjecture for elliptic curves with multiplicative reduction at $p$ which we obtain using bipartite Euler systems.
\end{abstract}

\maketitle

\tableofcontents

\section{Introduction}
 The goal of this article is to shed some light on the arithmetic intricacies of elliptic curves of \emph{rank two} using the mock plectic invariants introduced by Henri Darmon and the first author in \cite{DarmonFornea}. Our setting is complementary to that studied by Castella--Hsieh \cite{CastellaHsieh} and Castella \cite{CasetllaCM2} using generalized Kato classes because our elliptic curves are required to have multiplicative reduction at the chosen prime $p$.

\begin{assumptions}\label{assumptions1}
Let $p$ be a rational prime, $N$ a positive integer, $K$ a quadratic imaginary field where $p$ is inert, and $E_{/\Q}$ an elliptic curve of conductor $N$. Suppose that
    \begin{itemize}
    \item [$\bfcdot$] the conductor $N$ is unramified in $K$,
        \item [$\bfcdot$] the elliptic curve $E_{/\Q}$ has multiplicative reduction at $p$.
\end{itemize}
We can then write $N=p\hspace{0.1mm}\cdot\hspace{0.1mm} N^+\hspace{0.05mm}\cdot\hspace{0.05mm} N^-$ where $N^-$ denotes the largest divisor of $N/p$ which is divisible only by primes which are inert in $K$. To ensure that the root number $\varepsilon(E/K)$ of $E_{/\bb{Q}}$ over $K$ equals $+1$, we require that
 \begin{itemize}
        \item [$\bfcdot$] $N^-$ is square-free with an even number of prime factors.
    \end{itemize}
    For convenience, we also assume that $K\neq\Q(\sqrt{-1}),\Q(\sqrt{-3})$.
\end{assumptions}

In the first part of this article we generalize the construction in \cite{DarmonFornea} to associate a mock plectic invariant 
\[
\mathcal{Q}_K\in\mrm{H}^1(K,V_p(E))\otimes_{\Q_p}\mrm{H}_f^1(K_p,V_p(E)),
\]
to the triple $(E, K, p)$ (see eq.~\eqref{MPI}), where $V_p(E)$ denotes the rational $p$-adic Tate module of $E$.
By Proposition \ref{SelmerProp} the mock plectic invariant $\mathcal{Q}_K$ belongs to the tensor product of Bloch--Kato Selmer groups $\mrm{H}_f^1(K,V_p(E))\otimes_{\Q_p}\mrm{H}_f^1(K_p,V_p(E))$ whenever $L(E/K,1)=0$.
If, in addition, $\mathcal{Q}_K$ does not vanish, we state a conjecture (Conjecture \ref{Conj1}) predicting the rank of the Mordell--Weil groups $E(\Q)$ and $E(K)$.

In the second part of this paper we prove the weaker version of Conjecture \ref{Conj1} where ranks of Mordell--Weil groups are replaced by coranks of the corresponding $p$-primary Selmer groups.
A crucial step in the proof is establishing one inclusion of Perrin--Riou's Heegner point main conjecture for elliptic curves with multiplicative reduction at $p$ (see Theorem \ref{IMC}).
The argument is based on Howard's bipartite Euler systems machinery developed in \cite{BipartiteES}, while the construction of the Euler system per se is an adaptation of the groundbreaking work of Bertolini--Darmon \cite{BD} and the subsequent refinements by Pollack--Weston \cite{PW11}, Chida--Hsieh \cite{ChidaHsieh}, and Burungale--Castella--Kim \cite{BCK}. Let $G_\Q:=\Gal(\overline{\Q}/\Q)$ denote the absolute Galois group of $\Q$ and $\overline{\varrho}\colon G_\Q\to\mrm{GL}(E[p])$ the mod $p$ representation attached to $E_{/\Q}$.
\begin{assumptions}\label{assumptions2}
We require that
    \begin{itemize}
    \item $p\ge 5$
    \item $\overline{\varrho}$ is surjective,
    \item $\overline{\varrho}$ is ramified at every prime $\ell\mid N^-$ with $\ell^2\equiv1\pmod{p}$.
    \end{itemize}
\end{assumptions}
This latest set of hypotheses was first introduced by Pollack--Weston (see \cite{PW11}, Condition CR) and allows us to construct an Euler system \`a la Howard. It is key in ensuring the validity of Lemma \ref{completionCHAR} which controls the local structure of the character group of the toric part of certain Jacobians and is crucially used in the level raising part of the argument.
We also note that there are two steps in the construction of the bipartite Euler system that need to be treated with more care when the elliptic curve has multiplicative reduction at $p$. The proof of Lemma \ref{completionCHAR} uses Skinner--Zhang's level raising lemma (\cite{SkinnerZhang}, Lemma 5.4) when the residual mod $p$ representation $\overline{\varrho}$ is not finite at $p$. Furthermore, the proof of Proposition \ref{localPROP} uses Nekov\'a\v{r}'s analysis of the local properties of Kummer maps (see \cite{NekCanadian}, Proposition 2.7.12) to ensure that the cohomology classes satisfy the required local condition at $p$. 

\noindent To state our main result -- Theorem \ref{mainTHM} in the body of the article -- set $E^+ :=E$, denote by $E^-$ the quadratic twist of $E$ attached to $K$, write $\mrm{Sel}_{p^\infty}(E^\pm/\Q)$ for the $p$-primary Selmer groups of $E^\pm$ respectively, and define
\[
r_{p}(E^\pm/\Q):=\mrm{corank}_{\Z_p}\hspace{1mm}\mrm{Sel}_{p^\infty}(E^\pm/\Q),\qquad\quad \delta_p^\pm:=\begin{cases}
    1&\mbox{if}\ a_p(E^\pm)=+1\\
    0&\mbox{if}\ a_p(E^\pm)=-1.
    \end{cases}
\]
 Under the running Assumptions $\ref{assumptions1}\ \&\ \ref{assumptions2}$ we prove:
\begin{theorem}\label{MainThm}
 Suppose $L(E/K,1)=0$ and that $\mathcal{Q}_K\neq 0$, then
    \[
    r_{p}(E/K)= 2.
    \]
   Furthermore, the maximum of $\big\{r_{p}(E^\pm/\Q)+\delta_p^\pm\big\}$ is equal to $2$.
\end{theorem}
\begin{remark}
   The $p$-primary Selmer group $\mrm{Sel}_{p^\infty}(E^\pm/\Q)$ fits in the descent exact sequence
\[\xymatrix{
0\ar[r]& E^\pm(\Q)\otimes\Q_p/\Z_p\ar[r]& \mrm{Sel}_{p^\infty}(E^\pm/\Q)\ar[r]& \Sh(E^\pm/\Q)[p^\infty]\ar[r]& 0.
}\] 
Hence, if the $p$-power torsion of the Shafarevich--Tate group $\Sh(E^\pm/\Q)$ is finite, then
\[
r_{p}(E^\pm/\Q)=\mrm{rank}_\Z\hspace{1mm} E^\pm(\Q).
\]
\end{remark}

The simplest examples of triples $(E,K,p)$ that are expected to satisfy all the hypotheses stated thus far are found among elliptic curves of conductor $p$.
This happens because, when an elliptic curve $E_{/\Q}$ has conductor $p\ge11$, the mod $p$ Galois representation is surjective (\cite{Mazur-Goldfeld}, Theorem 4). Therefore, if $p$ is also inert in $K$, the only assumptions that might possibly not be met are the vanishing of the $L$-value $L(E/K,1)$ and the non-vanishing of the mock plectic invariant $\mathcal{Q}_K$.
In every given example concerning a semistable elliptic curve it is possible to determine whether the $L$-value vanishes or not with the aid of a computer algebra system like SageMath. 
However, we do not yet know how to check the non-vanishing of $\mathcal{Q}_K$ numerically.
 \begin{example}
 The elliptic curve
 \[
 E_{1/\Q}\colon\hspace{1mm} y^2+y=x^3-x\quad \mbox{of conductor } p=37 
 \]
 has non-split multiplicative reduction at $37$ and is the rational elliptic curve of minimal conductor among those of positive rank.
It has rank $1$ over $\Q$ and rank $2$ over the quadratic imaginary field $K=\Q(\sqrt{-2})$ where $37$ is inert.
As the sign of the functional equation $\varepsilon(E_1/\Q)$ equals $-1$, we also know that $L(E_1/K,1)=0$.
We expect $\mathcal{Q}_K$ to be non-trivial.
  \end{example}

 \begin{example}
 The unique elliptic curve of conductor $p=109$ is described by the  equation
 \[
 E_{2/\Q}\colon\hspace{1mm} y^2+xy=x^3-x^2-8x-7.
 \]
 It has rank $0$ over $\Q$ and attains rank $2$ over $K=\Q(\sqrt{-2})$ where $109$ is inert.
Using SageMath one can verify that the $L$-value $L(E_2/K,1)$ vanishes.
We expect $\mathcal{Q}_K$ to be non-trivial.
 \end{example}

\begin{example}
 The elliptic curve 
 \[
 E_{3/\Q}\colon\hspace{1mm} y^2+y=x^3+x^2+x+6
 \]
 of conductor $817=19\cdot 43$ has rank $2$ over $\Q$ and over $K=\Q(\sqrt{-7})$ where $19$ is inert and $43$ splits.
Using SageMath one can verify that the $L$-value $L(E_3/K,1)$ vanishes. If we take $p=19$, the residual mod $p$ representation is surjective and we expect $\mathcal{Q}_K$ to be non-trivial.
 \end{example}

\begin{acknowledgements}
We would like to thank Henri Darmon for many enlightening conversations related to this project.
We are also grateful to Francesc Castella, Matteo Longo and Matteo Tamiozzo for answering our questions, and to the anonymous
referee for their valuable comments and feedback.
The examples presented in this paper were found browsing LMFDB, and studied their properties using the computer algebra system SageMath. 
While working on this project Lennart Gehrmann has received funding from the Maria Zambrano Grant for the attraction of international talent in Spain. Michele Fornea was supported by the MTM grant PID2020-118236GB-I00 from Universitat Aut\`onoma de Barcelona, and the Maria de Maeztu grant CEX2020-001084-M  from the Centre de Recerca Matem\`atica.
\end{acknowledgements}

\section{Mock plectic invariants}
\subsection{CM points on Shimura curves}
Denote by $\overline{\Q}$ the algebraic closure of $\Q$ contained in $\C$.
Recall that $N=p\cdot N^+\cdot N^-$ with $N^-$ square-free and an even number of prime factors.
Let $B/\Q$ be the indefinite quaternion algebra of discriminant $N^-$ and $R\subseteq B$ an Eichler order of level $pN^+$.
Fix an isomorphism $\iota_\infty\colon B\otimes_\Q\R\xrightarrow{\sim}\mrm{M}_2(\R)$ which induces an embedding $B^\times\hookrightarrow \mrm{GL}_2(\R)$ and, therefore, an action of the elements of $B^\times$ with positive reduced norm on the Poincar\'e upper half plane $\mathcal{H}$. The Shimura curve associated to the Eichler order $R$ has a canonical model $Y$ defined over $\Q$ and a smooth compactification denoted by $X$.
Moreover, the complex manifold $Y(\C)$ is biholomorphic to the quotient $\Gamma\backslash \mathcal{H}$
where $\Gamma:=R^\times_1\subseteq R^\times$ is the subgroup of units of $R$ of reduced norm $1$.

\subsubsection{CM points.}
The group $B^\times$ acts on the set of $\Q$-algebra embeddings $\mrm{Emb}(K,B)$ of $K$ in $B$ by conjugation on the target.
By assumption, the prime divisors of $N^-$ are inert in the quadratic imaginary field $K$, thus $\mrm{Emb}(K,B)$ is non-empty.
Every element $\psi\in \mrm{Emb}(K,B)$ induces an action of $K^\times$ on $\bb{P}^1(\C)\setminus\bb{P}^1(\R)$ with a unique fixed point $z_\psi$ in the upper half plane $\mathcal{H}$, and the association
\[
\mrm{Emb}(K,B) \too \mathcal{H},\quad \psi\mapsto z_{\psi},
\]
is $B^\times$-equivariant.
It follows that the set of CM points of $X$ attached to $K$ 
\[
\mrm{CM}_K(X):=\big\{[z_\psi]\in X(\C)\ \lvert\ \psi\in\mrm{Emb}(K,B)\big\}
\]
is indexed by $\Gamma$-conjugacy classes of embeddings.
Given a positive integer $c$, let $\mathcal{O}_c:=\Z+c\mathcal{O}_K$ denote the order of $\mathcal{O}_K$ of conductor $c$.
The conductor of a CM point $[z_\psi]\in \mrm{CM}_K(X)$ is the positive integer $c$ such that  
\[
\psi(K)\cap R=\psi(\mathcal{O}_c),
\]
which does not depend on the choice of a representative of the $\Gamma$-conjugacy class of $\psi$.
\begin{remark}
Since $p$ divides the level of the Eichler order $R$ and is also inert in $K$, the conductor of any CM point of $X$ attached to $K$ is divisible by $p$ (see \cite[Lemma 2.2]{MumfordHeegner}).
\end{remark}
\noindent The theory of complex multiplication implies that $\mrm{CM}_K(X)\subseteq X(K^\mrm{ab})$, where $K^\mrm{ab}\subseteq \overline{\Q}$ is the maximal abelian extension of $K$, and the action of $\mrm{Gal}(K^\mrm{ab}/K)$ is explicitly described by Shimura's reciprocity law:
let $\widehat{\Z}$ be the profinite completion of $\Z$ and put $\widehat{A}:=A\otimes \widehat \Z$ for any algebra $A$.
Every embedding $\psi\colon K \hookrightarrow B$ induces an embedding $\widehat{\psi}\colon \widehat{K}\hookrightarrow \widehat{B}$ by extension of scalars.
If we normalize the reciprocity map of class field theory $\mrm{rec}_K\colon K^\times\backslash \widehat{K}^\times\to \Gal(K^\mrm{ab}/K)$ by matching local uniformizers with geometric Frobenius elements, the action of $\mrm{Gal}(K^\mrm{ab}/K)$ on CM points is given by
\begin{equation}\label{reciprocityLAW}
[z_{\psi}]^{\mrm{rec}_K(a)}=[g_{\psi,\alpha}\cdot z_\psi]\qquad\forall\ a\in\widehat{K}^\times
\end{equation}
where $g_{\psi,\alpha}\in B^\times$ is any element of positive reduced norm satisfying $g_{\psi,\alpha}^{-1}\cdot\widehat{\psi}(a)\in\widehat{R}^\times\subseteq \widehat{B}^\times$ (\cite{NekEuler}, Section 2). 
We denote by $\mrm{CM}_K^c(X)\subseteq \mrm{CM}_K(X)$ the subset of CM points of conductor $c$ and let $H_c/K$ be the abelian extension such that the global reciprocity maps induces an isomorphism
\[
\mrm{rec}_K\colon K^\times\backslash\widehat{K}^\times/\widehat{\mathcal{O}}_c^\times\overset{\sim}{\longrightarrow}\Gal(H_c/K).
\]
Shimura's reciprocity law implies that $\mrm{CM}_K^c(X)\subseteq X(H_c)$.

\subsubsection{Orbits for the action of $p$-arithmetic groups.}
Consider the $p$-arithmetic group 
\[
\Gamma^p:=R[1/p]^\times_1
\]
consisting of elements of reduced norm one of $R[1/p]$ and which acts on $\mathcal{H}$ with dense orbits.
For any embedding $\psi\colon K\hookrightarrow B$ 
we consider the $\Gamma^p$-orbit
\[
\Sigma_{\psi}:=\Gamma^p\cdot z_\psi \subseteq \mathcal{H}
\]
which depends only on the $\Gamma^p$-conjugacy class of $\psi$. The image $[\Sigma_{\psi}]$ of $\Sigma_{\psi}$ in $X(\C)$
consists of CM points whose conductor has a fixed prime-to-$p$ part $c$ and whose $p$-part can be arbitrarily large. In particular, we have the inclusion
\[
\left[\Sigma_\psi\right]\subseteq \bigcup_{n\ge0}\mrm{CM}_K^{cp^n}(X)
\]
telling us that $[\Sigma_{\psi}]\subset X(H_{cp^\infty})$ where  $H_{cp^\infty}:=\bigcup_{n\ge0} H_{c p^n}$ is a union of ring class fields.
For reference, we note here that the Galois group $\Gal(H_{cp^\infty}/K)$ sits in the short exact sequence 
\[\xymatrix{
1\ar[r]&\Gal(H_{cp^\infty}/H_c)\ar[r]&\Gal(H_{cp^\infty}/K)\ar[r]&\Gal(H_c/K)\ar[r]& 1,
}\]
and that class field theory provides a canonical isomorphism $\Gal(H_{cp^\infty}/H_c)\cong K_{p}^\times/\Q_p^\times$ showing, in particular, that $\Gal(H_{cp^\infty}/H_c)$ is pro-cyclic. 
Moreover, for any $n\ge1$ the Galois group $\Gal(H_{cp^n}/H_c)$ is seen to be the unique quotient of $\Gal(H_{cp^\infty}/H_c)$ of order $(p+1)p^{n-1}$.

\subsubsection{Mock Shimura surfaces.}
Fix an isomorphism $\iota_p\colon B\otimes_\Q\Q_p\xrightarrow{\sim}\mrm{M}_2(\Q_p)$ which induces an embedding $B^\times\hookrightarrow \mrm{GL}_2(\Q_p)$ and, thus, an action of $B^\times$ on Drinfeld's $p$-adic upper-half plane $\mathcal{H}_p=\bb{P}^1(\C_p)\setminus\bb{P}^1(\Q_p)$.
The mock Shimura surface attached to the triple of integers $(p,N^+,N^-)$ is defined to be the quotient
\[
\mathcal{S}:=\Gamma^p\backslash (\mathcal{H}\times\mathcal{H}_p)
\]
where $\Gamma^p$ acts diagonally via the fixed isomorphisms $\iota_\infty$ and $\iota_p$.
Since the prime $p$ is inert in $K$, the action of $K^\times$ on $\mathcal{H}_p$ induced by any embedding $\psi\colon K\hookrightarrow B$ has two fixed points $\tau_\psi$, $\overline{\tau}_\psi$, which are interchanged by the Galois group $\Gal(K_p/\Q_p)=\langle\sigma_p\rangle$.
Denote by $\tau_{\psi}\in\mathcal{H}_p$ the fixed point such that $K^\times$ acts on its tangent space through the character $k\mapsto k^{1-\sigma_p}$.
We define the set of CM points of $\mathcal{S}$ attached to $K$ as
\[
\mrm{CM}_K(\mathcal{S}):=\big\{s_{\psi}:=[z_\psi, \tau_\psi]\in \mathcal{S}\ \lvert\ \psi\in\mrm{Emb}(K,B)\big\},
\]
which is indexed by $\Gamma^p$-conjugacy classes of embeddings. 
The conductor of a point $s_{\psi}\in \mrm{CM}_K(\mathcal{S})$ is defined to be the positive prime-to-$p$ integer $c$ such that  
\[
\psi(K)\cap R[1/p]=\psi(\mathcal{O}_c[1/p]).
\]
We write $\mrm{CM}_K^c(\mathcal{S})$ for the set of CM points on $\mathcal{S}$ of conductor $c$. 
\begin{remark}\label{bijectionCM}
The map associating the CM point $s_{\psi}\in \cal{S}$ to the $\Gamma^p$-orbit $\Sigma_{\psi}\subset\cal{H}$ is a bijection that preserves the prime-to-$p$ part of the conductor. 
\end{remark}

\subsubsection{Torus actions on the Bruhat--Tits tree.}\label{actionontree}
Let $\scr{T}$ denote the Bruhat--Tits tree of the group $\scr{G}:=\mrm{PGL}_2(\Q_p)$ whose set  of vertices $\scr{V}$ is given by the set of maximal $\Z_p$-orders in $\mrm{M}_2(\Q_p)$. Two vertices are connected by an edge if and only if the intersection of the corresponding maximal orders is an Eichler order of level $p$.
The group $\scr{G}$ acts transitively on $\scr{V}$ and on the set $\scr{E}$ of oriented edges of $\scr{T}$.
Furthermore, there exists a unique vertex $v_\circ\in\scr{V}$ with stabilizer $\scr{K}:=\mrm{PGL}_2(\Z_p)$, and a unique oriented edge $e_\circ\in\scr{E}$ with source $s(e_\circ)=v_\circ$ whose
stabilizer is the Iwahori subgroup $\scr{I}\subseteq \scr{K}$ of upper triangular matrices mod $p$.
This data determines the following  isomorphisms of $\scr{G}$-sets:
\begin{equation}\label{group-uniformization}
\scr{V}\cong  \scr{G}/\scr{K}\quad\mbox{and}\quad \scr{E}\cong \scr{G}/\scr{I}.
\end{equation}
Let $\scr{G}^0\subseteq \scr{G}$ be the kernel of the homomorphism
\[
\scr{G}\longrightarrow \{\pm 1\},\quad g \mapsto (-1)^{\ord_p(\det(g))},
\]
and define the sets of even vertices and even oriented edges by
\[
\scr{V}^0:=\scr{G}^0 \cdot v_\circ\quad\mbox{and}\quad\scr{E}^0:=\scr{G}^0 \cdot e_\circ.
\]
The $p$-arithmetic group $\Gamma^p$ acts transitively on  $\scr{V}^0$ and $\scr{E}^0$ by \cite[Proposition 3.2]{PlecticInvariants}, hence there are bijections
\begin{equation}\label{GammaIdentifications}
\scr{V}^0 \cong \Gamma^p/\Gamma'\quad\mbox{and}\quad \scr{E}^0\cong \Gamma^p/\Gamma,
\end{equation}
where $\Gamma'=\Gamma^p\cap \scr{K}$ is the group of reduced norm one elements of the unique Eichler order of level $N^+$ containing $R$.

\smallskip
\noindent Through the fixed isomorphism $\iota_p\colon B\otimes_\Q\Q_p\xrightarrow{\sim}\mrm{M}_2(\Q_p)$, any embedding $\psi\colon K\hookrightarrow B$ induces an action of $K_p^\times/\Q_p^\times$ on the Bruhat-Tits tree $\scr{T}$.
Since $p$ is inert in $K$, this action fixes a unique vertex $v_\psi\in\scr{V}$ and $\scr{U}_0:=K_p^\times/\Q_p^\times$ acts transitively on the set of vertices at fixed distance from $v_\psi$ (\cite{CornutVatsal}, Lemma 6.5).
For every vertex $v_n$ at distance $n\ge1$ from $v_\psi$ the stabilizer
\[\scr{U}_n:=\mrm{Stab}_{K_p^\times/\Q_p^\times}(v_n)
\]
is the unique subgroup such that $\scr{U}_0/\scr{U}_n$ is cyclic of degree $(p+1)p^{n-1}$.
Let $\{e_{n}^\psi\}_{n\ge1}$ be a sequence of consecutive edges originating from $v_\psi$, i.e, $s(e_{1}^\psi)=v_\psi$ and $t(e_{n}^\psi)=s(e_{n+1}^\psi)$ for every $n\ge1$. We deduce the following useful properties:
\begin{itemize}
    \item $\mrm{Stab}_{K_p^\times/\Q_p^\times}(e_n^\psi)=\scr{U}_n$,
     \item $\scr{U}_n/\scr{U}_{n+1}$ acts simply transitively on the edges $e\neq\overline{e}_n^\psi$ that satisfy $s(e)=t(e_n^\psi)$,
     \item and $K_p^\times/\Q_p^\times$ acts simply transitively on sequences of consecutive edges originating from $v_\psi$.
\end{itemize}

\begin{remark}
Since the group $\Gamma^p$ acts transitively on the set of unoriented edges of $\scr{T}$ by \eqref{GammaIdentifications}, it follows that any $\Gamma^p$-orbit of embeddings in $\mrm{Emb}(K,B)$ contains an element whose conductor has $p$-adic valuation equal to $1$.
\end{remark}

\subsection{Construction}
We introduce the Steinberg representation of $\mathscr{G}$ to describe an enhancement of modular parametrizations of the elliptic curve $E$ by the Shimura curve $X$.

\subsubsection{Steinberg representation.}\label{Steinberg}
Let $\St(\Z)$ denote the space of locally constant $\Z$-valued function on $\bb{P}^1(\Q_p)$ modulo constants. We record here the well-known and useful fact about the existence of a basis $\{U_e\}_{e\in\scr{E}}$ for the $p$-adic topology of $\bb{P}^1(\Q_p)$ indexed by oriented edges of the Bruhat--Tits tree satisfying the following properties
\begin{itemize}
\item $U_{\overline{e}_\circ}=\Z_p$,

\vspace{1mm}
\item $U_e=\bb{P}^1(\Q_p)\setminus U_{\overline{e}}$\hspace{1.5mm} for all $e\in\scr{E}$,

\vspace{1mm}
\item $g\cdot U_e=U_{g\cdot e}$\hspace{1.5mm} for all $g\in \scr{G}$, $e\in\scr{E}$,

\vspace{1mm}
\item $\coprod_{s(e)=v}U_e=\bb{P}^1(\Q_p)$\hspace{1.5mm} for all $v\in\scr{V}$.
\end{itemize}
\begin{remark}\label{specialCovers}
  A direct consequence of the material presented in Sections \ref{actionontree} and \ref{Steinberg} is that for an embedding $\psi\colon K\hookrightarrow B$, a sequence of consecutive edges $\{e_{n}^\psi\}_{n\ge1}$, and $n\ge1$ the set
	\[
	\mathcal{C}_n^\psi:=\{\psi(\alpha)\cdot U_{e_n^\psi}\ \vert\ \alpha \in K_{p}^\times\}
	\]
	is a covering of $\bb{P}^1(\Q_p)$ by disjoint compact open subsets. Furthermore, the coverings $\{\mathcal{C}_{n}^\psi\}_{n\ge1}$ form a cofinal subset of the partially ordered set of all coverings of $\bb{P}^1(\Q_p)$.
\end{remark}
 The map $e\mapsto \overline{e}$ which switches source and target of an oriented edge $e\in\scr{E}$ induces an involution $w_p$ on the free abelian group $\Z[\scr{E}]$.
Set $\Z[\scr{E}]_+:=\Z[\scr{E}]/(1+w_p)\Z[\scr{E}]$ and recall the $\scr{G}$-equivariant short exact sequence 
\begin{equation}\label{resolutionST}
\xymatrix{
0\ar[r]& \Z[\scr{V}]\ar[r]^-{\delta}&\Z[\scr{E}]_+\ar[r]^-{\mrm{Ev}}\ar[r]& \St(\Z)\ar[r]& 0
}\end{equation}
where 
\[
\delta(v)(e)=\sum_{s(e)=v} e\quad \mbox{and}\quad \mrm{Ev}(e)=\mathbbm{1}_{U_e}.
\]
For reference we note that the natural map $\Z[\scr{E}^0]\to\Z[\scr{E}]_+$ is a $\scr{G}^0$-equivariant isomorphism.
\begin{remark}\label{geometricOrigin}
Consider the natural projection $\pi\colon \scr{G}/\scr{I}\to \scr{G}/\scr{K}$.
Under the identifications in \eqref{group-uniformization}, the homomorphism $\delta\colon \Z[\scr{V}]\to \Z[\scr{E}]$ corresponds to
\[
\Z\big[ \scr{G}/\scr{K}\big]\longrightarrow \Z\big[\scr{G}/\scr{I}\big],\quad h\mapsto \sum_{\pi(g)=h}[g].
\]
\end{remark}

\subsubsection{Enhanced modular parametrization.}
Let $J_X$ denote the Jacobian of the Shimura curve $X$.
Modularity of elliptic curves over $\Q$, the Jacquet--Langlands correspondence, and Faltings' isogeny theorem imply the existence of a surjective $\Q$-rational homomorphism $\varphi\in \mrm{Mor}(J_X,E)$ of abelian varieties.
The choice of an auxiliary prime $\ell\nmid N$ such that $a_{\ell}(E)\not\equiv\ell+1\pmod{p}$ allows us to consider the map 
\[
\alpha\colon X\longrightarrow J_X\otimes\Z_{p},\quad x\mapsto (T_{\ell}-\ell-1)[x]\otimes(a_{\ell}(E)-\ell-1)^{-1}.
\]
The composition $\varphi\circ\alpha\colon X\to  E\otimes\Z_p$ is independent of the choice of $\ell$ and gives an element
\[
\varphi_\C\in \Hom_\Z\big(\Z[Y(\C)],\hspace{0.5mm} E(\C)\otimes\Z_p\big)=\Hom_\Gamma(\Z[\mathcal{H}],\hspace{0.5mm} E(\C)\otimes\Z_p\big).
\]
Since the elliptic curve $E_{/\Q}$ has multiplicative reduction at $p$, we can upgrade $\varphi_\C$ as follows.
\begin{proposition}\label{constructionMeasure}
There exists a unique class 
\[
c_\varphi\in \Hom_{\Gamma^p}(\bb{Z}[\mathcal{H}]\otimes \St(\Z),\hspace{0.5mm} E(\C)\otimes\Z_p\big)
\]
such that
\begin{itemize}
\item $c_\varphi(z \otimes e_\circ)=\varphi_\C(z)$ for all $z\in\mathcal{H}$ and
\item $c_\varphi(z \otimes e)=-c_\varphi(z \otimes \overline{e})$ for all $e\in \scr{E}$, $z\in\mathcal{H}$.
\end{itemize}
\end{proposition}
\begin{proof}
Let us consider the $\Gamma^p$-module $M:=\Hom_\Z(\Z[\mathcal{H}],\hspace{0.1mm} E(\C)\otimes\Z_p)$.
Applying the functor $\Hom_{\Gamma^p}(-,M)$ to the short exact sequence \eqref{resolutionST} and using the isomorphism $\Z[\scr{E}]_+\cong\Z[\scr{E}^0]$ we extract the following exact sequence 
 \begin{equation}\label{pieceofles}
\xymatrix{
0\ar[r]& \Hom_{\Gamma^p}(\St(\Z),M)\ar[r]& \Hom_{\Gamma^p}(\Z[\scr{E}^0], M)\ar[r]^-{\delta^{\ast}}&
\Hom_{\Gamma^p}(\Z[\scr{V}], M).
}\end{equation}
Now, by \eqref{GammaIdentifications} and Frobenius reciprocity, we can rewrite \eqref{pieceofles} as 
 \begin{equation}\label{pieceofles2}
\xymatrix{0\ar[r]& \Hom_{\Gamma^p}(\St(\Z),M)\ar[r]& \Hom_{\Gamma}(\Z, M)\ar[r]^-{\delta^{\ast}}&
\Hom_{\Gamma'}(\Z, M)^{2},
}\end{equation}
where $\Gamma'$ is the group of units of reduced norm one in an Eichler order of level $N^+$.
By construction $\varphi_\C$ is an element of $\Hom_{\Gamma}(\Z, M)$.
Thus, to prove the claim, we have to show that  $\delta^{\ast}(\varphi_\C)=0$.
We have equalities
\[
\Hom_{\Gamma}(\Z, M)=\Hom_\Z(\Z[Y(\C)],\hspace{0.5mm} E(\C)\otimes\Z_p),\quad \Hom_{\Gamma'}(\Z, M)=\Hom_\Z(\Z[Y'(\C)],\hspace{0.5mm} E(\C)\otimes\Z_p)
\]
where $Y'$ denotes the Shimura curve of level $\Gamma'$.
Let $J_{X'}$ denote the Jacobian of the compactification $X'$ of $Y'$ and
$\mrm{pr}_{1},\mrm{pr}_{2}\colon J_X\to J_{X'}$ the two standard degeneracy maps, then
Remark \ref{geometricOrigin} implies that the following diagram commutes
\[\xymatrix{
\mrm{Mor}(J_X,E)\ar[rr]^-{(\mrm{pr}_1\oplus\mrm{pr}_2)^*}\ar[d]_{\alpha^*}&&\mrm{Mor}(J_{X'},E)^{2}\ar[d]_{\alpha^*}\\
\Hom_\Z(\Z[Y(\C)],\hspace{0.5mm} E(\C)\otimes\Z_p)\ar[rr]^-{\delta^*}&& \Hom_\Z(\Z[Y'(\C)],\hspace{0.5mm} E(\C)\otimes\Z_p)^{2}.
}\]
As $E$ has multiplicative reduction at $p$, $\mrm{Mor}(J_{X'},E)=0$. Thus, we obtain $\delta^\ast(\varphi_\C)=0$.
\end{proof}
 
\subsubsection{$p$-adic integration.}
For any abelian group $M$ it is possible to identify $\Hom_\Z(\St(\Z), M)$ with the space of $M$-valued distributions of total mass $0$ on $\bb{P}^1(\Q_p)$.
Therefore, the class $c_\varphi$ of Proposition \ref{constructionMeasure} determines a $\Gamma^p$-invariant $\Hom_\Z\big(\Z[\mathcal{H}], E(\C)\otimes\Z_p\big)$-valued measure $\mu_\varphi$ on $\bb{P}^1(\Q_p)$.
To have an interesting $p$-adic integration theory we need our measures to be valued in $p$-adically complete and separated $\Z_p$-modules, but the $p$-adic completion of $E(\C)\otimes\Z_p$ is trivial because the group is divisible. To resolve this issue, we consider subsets of $\mathcal{H}$ which are preserved by the action of $\Gamma^p$ and whose images under modular parametrizations are defined over controlled algebraic extension of $\Q$. 

\smallskip
\noindent Let $c$ be a positive integer prime to $N$ unramified in $K$, and consider $\psi\in\mrm{Emb}(K,B)$ whose conductor has prime-to-$p$ part equal to $c$.
The theory of complex multiplication ensures that the image of the $\Gamma^p$-orbit $\Sigma_{\psi}\subset\cal{H}$ under a modular parametrization $\varphi_\C$ belongs to $E(H_{cp^\infty})$.
Therefore, from Proposition \ref{constructionMeasure} we obtain a class
\[
c_{\psi}\in \Hom_{\Gamma^p}\big(\bb{Z}[\Sigma_{\psi}]\otimes \St(\Z),\hspace{0.5mm} E(H_{cp^\infty})\otimes\Z_p\big).
\]
As there are only finitely many torsion points in $E(H_{cp^\infty})$ by \cite[Lemma 6.3]{BD-Edix}, \cite[Proposition 7]{freeMW} implies that the quotient of $E(H_{cp^\infty})$ by its torsion subgroup is a free abelian group.
It follows that the natural map from $E(H_{cp^\infty})\otimes\Z_p$ to its $p$-adic completion $E(H_{cp^\infty})^{\wedge}_p$ is injective. Therefore, we do not lose information by considering 
\begin{equation}
    c_{\psi}\in \Hom_{\Gamma^p}\big(\bb{Z}[\Sigma_{\psi}]\otimes \St(\Z),\hspace{0.5mm}E(H_{cp^\infty})^{\wedge}_p\big).
\end{equation}
For a $p$-adically separated and complete $\Z_p$-module $M$ we denote by
\[
\St^\mrm{ct}(M):=C(\bb{P}^1(\Q_p),M)/ M
\]
 the continuous Steinberg representation with coefficients in $M$. 
Approximating continuous functions by locally constant ones determines a homomorphism
\[
\Hom_\Z(\St(\Z),N)\too \Hom_\Z(\St^{\mrm{ct}}(M),N\otimes_{\Z_p} M)
\]
for any other $p$-adically separated and complete $\Z_p$-module $N$ (cf.~\cite[Lemma 4.2]{plecticHeegner}).
Specializing the discussion to the setting where $M=\widehat{K}^\times_{p,1}$ is the $p$-adic completion of the group of norm one element in $K_p^\times$ and $N= E(H_{cp^\infty})^{\wedge}_{p}$ yields the map
\[
\Hom_{\Gamma^p}\big(\bb{Z}[\Sigma_{\psi}]\otimes \St(\Z),\hspace{0.8mm} E(H_{cp^\infty})^{\wedge}_p\big)
\too
\Hom_{\Gamma^p}\big(\bb{Z}[\Sigma_{\psi}]\otimes \St^{\mrm{ct}}(\widehat{K}_{p,1}^\times),\hspace{0.8mm} E(H_{cp^\infty})^{\wedge}_p\otimes_{\Z_p}\widehat{K}_{p,1}^\times\big).
\]
Concretely, this means that we can integrate continuous $\widehat{K}_{p,1}^\times$-valued functions with respect to the measure attached to $c_{\psi}$. The embedding $\psi$ naturally gives rise to the function
\[
f_\psi\colon\bb{P}^{1}(\Q_p)\too \widehat{K}_{p,1}^\times,\quad t\mapsto \frac{t-\tau_\psi}{t-\overline{\tau}_\psi}.
\]

\begin{definition}\label{def:mpi}
The mock plectic invariant attached to $\psi$ is the element 
\[
\mathcal{Q}_{\psi}:=c_{\psi}(z_\psi \otimes f_\psi)\ \in E(H_{cp^\infty})^{\wedge}_{p}\otimes_{\Z_p}\widehat{K}_{p,1}^\times.
\]
\end{definition}
\noindent The mock plectic invariant $\mathcal{Q}_{\psi}$ depends only on the $\Gamma^p$-conjugacy class of $\psi$ and the choice of modular parametrization $\varphi$.
This produces a well-defined function
\[
\mrm{AJ}_{\mathcal{S}}^{-}\colon \mrm{CM}_K^c(\mathcal{S})\too E(H_{cp^\infty})^{\wedge}_{p}\otimes_{\Z_p}\widehat{K}_{p,1}^\times,
\qquad s_\psi\mapstoo \mathcal{Q}_{\psi}.
\]

\subsubsection{Kolyvagin derivatives.}\label{Kolder}
We can also express the mock plectic invariant $\mathcal{Q}_{\psi}$ as the limit of Kolyvagin derivatives of a trace compatible system of Heegner points.
For this, let $\{e_n^\psi\}_{n\ge1}$ be a sequence of consecutive oriented edges originating from $v_\psi$ (see Section \ref{actionontree}) in the Bruhat--Tits tree $\scr{T}$.
For every $n\ge1$ we consider the point 
\begin{equation}\label{thepoints}
{^{\psi}}\hspace{-.5mm}P_{n}:=c_{\psi}(z_\psi\otimes e_n^\psi)\ \in\ E(H_{cp^\infty})\otimes\Z_p.
\end{equation}
\begin{lemma}\label{compatiblePOINTS}
For every $n\ge1$ we have ${^{\psi}}\hspace{-.5mm}P_{n}\in E(H_{cp^n})\otimes\Z_p$.
Moreover, the equalities 
  \[
 \mrm{Trace}_{H_{cp^{n+1}}/H_{cp^n}}({^{\psi}}\hspace{-.5mm}P_{n+1})={^{\psi}}\hspace{-.5mm}P_{n}\qquad\&\qquad\mrm{Trace}_{H_{cp/H_{c}}}({^{\psi}\hspace{-.5mm}P_{1}})=0
  \]
  hold.
\end{lemma}
\begin{proof}
  Using the explicit description of $c_{\psi}$ given in Proposition \ref{constructionMeasure} and Shimura's reciprocity law \eqref{reciprocityLAW}, one sees that ${^{\psi}}\hspace{-.5mm}P_{n}\in E(H_{cp^\infty})\otimes\Z_p$ is fixed by the stabilizer $\scr{U}_n\cong \Gal(H_{cp^\infty}/H_{cp^n})$ of $e_n^\psi$ in $K_p^\times/\Q_p^\times\cong \Gal(H_{cp^\infty}/H_c)$.
Hence, ${^{\psi}}\hspace{-.5mm}P_{n}$ belongs to $E(H_{cp^n})\otimes\Z_p$. 
   As $ \scr{U}_n/\scr{U}_{n+1}\cong \Gal(H_{cp^{n+1}}/H_{cp^n})$ has order $p$ and acts simply transitively on the edges $e\neq\overline{e}_n^\psi$ such that $s(e)=t(e_n^\psi)$, we compute that
  \[\begin{split}
  \mrm{Trace}_{H_{cp^{n+1}}/H_{cp^n}}({^\psi}\hspace{-.5mm}P_{n+1})&=\sum_{\alpha\in \scr{U}_n/\scr{U}_{n+1}}({^{\psi}\hspace{-.5mm}P_{n+1}})^{\mrm{rec}_K(\alpha)}\\
  &=\sum_{ e\neq\overline{e}_n^\psi, \hspace{0.5mm} s(e)=t(e_n^\psi)}c_{\psi}(z\otimes e).
  \end{split}\]
In the Steinberg representation $\St(\Z)$ the equalities $\sum_{s(e)=t(e_n^\psi)}\mathbbm{1}_{U_e}=0$ and $\mathbbm{1}_{U_{e_n^\psi}}=-\mathbbm{1}_{U_{\overline{e}_n^\psi}}$ hold and, thus, we compute
  \[\begin{split}
  \mrm{Trace}_{H_{cp^{n+1}}/H_{cp^n}}({^{\psi}\hspace{-.5mm}P_{n+1}})&=-c_{\psi}(z\otimes\overline{e}_n^\psi)\\
  &={^{\psi}}\hspace{-.5mm}P_{n}.
  \end{split}\]
  Similarly, since $\Gal(H_{cp}/H_c)$ has order $p+1$ and acts simply transitively on the edges $e\in\scr{E}$ with $s(e)=s(e_1^\psi)$, we deduce that $\mrm{Trace}_{H_{cp/H_{c}}}({^{\psi}\hspace{-.5mm}P_{1}})=0$.
\end{proof}

Let now $\mu_{\psi}$ be the measure attached to the class $c_{\psi}(z_\psi \otimes (-))$.
It is a direct consequence of the construction that
\[
\mathcal{Q}_{\psi}=\mint_{\bb{P}^1(\Q_p)}\left(\frac{t-\tau_\psi}{t-\overline{\tau}_\psi}\right)\mrm{d}\mu_{\psi}(t).
\]
The M\"obius transformation $A_\psi$ that sends $(\tau_\psi,\overline{\tau}_\psi,\infty)$ to $(0,\infty,1)$ induces a homeomorphism from $\bb{P}^1(\Q_p)$ to $K_{p,1}^\times$.
By changing coordinates in the integral through $A_\psi$ we can write
\[
\mathcal{Q}_{\psi}=\mint_{K_{p,1}^\times}\xi\ \mrm{d}\widetilde{\mu}_{\psi}(\xi)
\]
where $\widetilde{\mu}_{\psi}$ denotes the pushforward of the measure $\mu_{\psi}$ along $A_\psi$.
For any $n\ge1$, the covering $\mathcal{C}_n^\psi$ of $\bb{P}^1(\Q_p)$ described in Remark \ref{specialCovers} corresponds, under the homeomorphism $A_\psi$, to  $\{\alpha\cdot\scr{U}_n\}_{\alpha\in K_{p}^\times}$.
Therefore, the mock plectic invariant $\mathcal{Q}_{\psi}$ equals the limit 
\begin{equation}\label{KolyDerivative}
\mathcal{Q}_{\psi}=\underset{n\to\infty}{\lim}\ \sum_{\alpha\in \scr{U}_0/\scr{U}_n}({^{\psi}\hspace{-.5mm}P_{n}})^{\mrm{rec}_K(\alpha)}\otimes \alpha
\end{equation}
of Kolyvagin derivatives of Heegner points, where we used the identification $K_p^\times/\Q_p^\times\cong K_{p,1}^{\times}$ given by $x\mapsto x^{1-\sigma_p}$.

\begin{corollary}\label{eigenspacePOS}
The mock plectic invariant is fixed by the action of $G_{H_c}$, that is, 
\[
\mathcal{Q}_{\psi}\in \left(E(H_{cp^\infty})^{\wedge}_{p} \otimes_{\Z_p} \widehat{K}_{p,1}^\times\right)^{G_{H_c}}.
\]
Moreover, it belongs to the $-a_p(E)\cdot\varepsilon(E/\Q)$ eigenspace for the action of complex conjugation.
\end{corollary}
\begin{proof}
 Using the expression \eqref{KolyDerivative} of  $\mathcal{Q}_{\psi}$ in terms of Kolyvagin derivatives, the first claim follows from the fact that $\mrm{Trace}_{H_{cp^n}/H_c}({^{\psi}\hspace{-.5mm}P_{n}})=0$  for any $n\ge1$ by Lemma \ref{compatiblePOINTS}, and the formula \cite[Equation 3.5]{Gross}.
The second claim is a special case of \cite[Corollary 2.15(2)]{MumfordHeegner}.
\end{proof}

\subsubsection{Image in Galois cohomology.}
As $\widehat{K}_{p,1}^\times$ is a free $\Z_p$-module of rank one, we deduce from Corollary \ref{eigenspacePOS} and flat base change that the mock plectic invariant belongs to
    \[
\mathcal{Q}_{\psi}\in \left(E(H_{cp^\infty})^{\wedge}_{p}\right)^{G_{H_c}} \otimes_{\Z_p} \widehat{K}_{p,1}^\times.
\]
We can now relate $\mathcal{Q}_{\psi}$ to the arithmetic of $E$ over $H_c$ using Kummer maps. More precisely, when the mod $p$ Galois representation attached to $E_{/\Q}$ is surjective, one can use the Kummer maps 
\[
\big(E(H_{cp^{m}})/p^{n}\big)^{G_{H_c}}\longrightarrow\mrm{H}^1(H_{cp^{m}}, E[p^n])^{G_{H_c}},
\]
together with the isomorphisms $\mrm{H}^1(H_{cp^{m}}, E[p^n])^{G_{H_c}}\cong \mrm{H}^1(H_c, E[p^n])$ arising from the inflation-restriction exact sequence (\cite{DarmonFornea}, Lemma 3.6), to produce a homomorphism
\begin{equation}\label{globalkummer}
\left(E(H_{cp^\infty})^{\wedge}_{p}\right)^{G_{H_c}} \too \mrm{H}^{1}(H_c,V_p(E)).
\end{equation}
Moreover, since $E$ has multiplicative reduction at $p$, Tate's uniformization and the local Kummer map induce an injection 
\begin{equation}\label{localkummer}
\widehat{K}_{p,1}^{\times} \longhookrightarrow \mrm{H}^{1}_f(K_p,V_p(E))^{-a_p}
\end{equation}
where $\mrm{H}^{1}_f(K_p,V_p(E))^{-a_p}\subseteq \mrm{H}^{1}_f(K_p,V_p(E))$ denotes the subspace on which the non-trivial element $\sigma_p$ of $\mrm{Gal}(K_p/\Q_p)$ acts via multiplication with $-a_p(E)$.
The tensor product of the homomorphisms \eqref{globalkummer} and \eqref{localkummer} gives
\begin{equation}\label{globallocalkummer}
\left(E(H_{cp^\infty})^{\wedge}_{p}\right)^{G_{H_c}} \otimes_{\Z_p} \widehat{K}_{p,1}^\times
\too 
\mrm{H}^1(H_c,V_p(E))\otimes_{\Q_p}\mrm{H}_f^1(K_p,V_p(E))^{-a_p}.
\end{equation}
By a slight abuse of notation we will denote also by  $\mathcal{Q}_\psi$ the image of the mock plectic invariant under the homomorphism of \eqref{globallocalkummer}.

\medskip
\noindent Set $\Delta_c:=\Gal(H_c/K)$ and consider a primitive character $\chi\colon \Delta_c \to\C^\times$.
Let $\Q^\chi_p$ be the finite extension of $\Q_p$ generated by the values of $\chi$ and denote by $e_\chi\in\Q^\chi_p[\Delta_c]$ the idempotent satisfying $e_\chi[\sigma]=\chi(\sigma)e_\chi$.
For any $\Q_p[\Delta_c]$-module $M$ we set $M^\chi:=e_\chi (M\otimes_{\Q_p}\Q_p^\chi)$ and define
\[
\mathcal{Q}^\chi:=e_\chi \mathcal{Q}_{\psi}\hspace{.7mm}\in\hspace{.7mm} \mrm{H}^1\big(H_c, V_p(E)\big)^\chi \otimes_{\Q_p}\mrm{H}_f^1(K_p,V_p(E))^{-a_p}.
\]
\begin{remark}
    Mock plectic invariants belong to the tensor product of \emph{global} and local Galois cohomology groups, while the plectic Stark--Heegner points studied in \cite{PlecticInvariants}, \cite{plecticHeegner}, \cite{polyquadraticPlectic} could only be constructed as elements in tensor products of local Galois cohomology groups. It seems likely that -- at least for quadratic CM extensions -- the techniques developed in \cite{DarmonFornea} and in this paper will also allow to construct canonical lifts of plectic Stark--Heegner points to tensor products of global and local Galois cohomology groups.
The authors hope to return to this question in the future.

\end{remark}
The next proposition gives a criterion for the mock plectic invariant to be a Selmer class in terms of the vanishing of an appropriate $L$-value.
\begin{proposition}\label{SelmerProp}
    We have 
    \[
   \mathcal{Q}^\chi\hspace{.5mm}\in\hspace{.5mm} \mrm{H}^1_f\big(H_c, V_p(E)\big)^\chi\otimes_{\Q_p}\mrm{H}_f^1(K_p,V_p(E))^{-a_p}\quad\iff\quad L(E/K,\chi,1)=0.
    \]
\end{proposition}
\begin{proof}
    It suffices to check the local condition at the primes above $p$. The claim is then reduced to \cite[Theorem 4.3]{BD-Edix} as in the proof of \cite[Theorem 3.8]{DarmonFornea}.
\end{proof}
\begin{remark}
     When the $L$-value $L(E/K,\chi,1)$ is non-zero, the mock plectic invariant $\mathcal{Q}^\chi$ is not a Selmer class by Proposition \ref{SelmerProp}, and can be used to prove instances of the BSD-conjecture in rank $0$ as explained in Theorem 3.9 and Remark 3.10 of \cite{DarmonFornea}.
\end{remark}

   \noindent  When $c=1$ and $\chi=\mathbbm{1}$ is the trivial character $ \mrm{H}^1(H_1, V_p(E))^{\mathbbm{1}}\cong \mrm{H}^1(K, V_p(E))$.
    In this case we denote the mock plectic invariant $\mathcal{Q}^{\mathbbm{1}}$ by 
    \begin{equation}\label{MPI}
    \mathcal{Q}_K\in \mrm{H}^1(K, V_p(E))\otimes_{\Q_p}\mrm{H}_f^1(K_p,V_p(E))^{-a_p}.
    \end{equation}

\subsection{Conjectures}
We extend the conjectures of \cite[Section 3.8]{DarmonFornea} to the setting considered in this paper.
Put $E^+ :=E$, denote by $E^-$ the quadratic twist of $E$ attached to $K$, and define
\[
\delta_p^\pm:=\begin{cases}
    1&\mbox{if}\ a_p(E^\pm)=+1\\
    0&\mbox{if}\ a_p(E^\pm)=-1.
    \end{cases}
\]
For a positive integer $c$ prime to $N$ and  a primitive character $\chi\colon\Delta_c\to\C^\times$ we set
\[
r_\mrm{alg}(E/K,\chi):=\mrm{dim}_\C\hspace{1mm} \mrm{H}^0\big(K, E(H_c)\otimes\C(\chi^{-1})\big).
\]
Since the extension $H_c/\bb{Q}$ is generalized dihedral, when $\chi\colon\Delta_c\to\C^\times$ is quadratic its induction is reducible. There are quadratic characters $\chi_\pm\colon\mrm{Gal}(H_c/\bb{Q})\to\bb{C}^\times$ such that
\[
\mrm{Ind}_{G_K}^{G_\bb{Q}}(\chi)=\chi_+\oplus\chi_-
\]
and which are distinguished by their value $\chi_\pm(\mrm{Fr}_p)=\pm1$ on the conjugacy class of Frobenius at $p$.
The quantities
\[
r_\mrm{alg}^\pm(E/K,\chi):=\mrm{dim}_\C\hspace{1mm} \mrm{H}^0\big(\bb{Q}, E(H_c)\otimes\C(\chi_\pm)\big)
\]
are then related by the equality $r_\mrm{alg}(E/K,\chi)=r_\mrm{alg}^+(E/K,\chi)+r_\mrm{alg}^-(E/K,\chi)$.
\begin{conjecture}\label{Conj1}
  Suppose that  $L(E/K,\chi,1)= 0$. If $\chi\colon\Delta_c\to\C^\times$ is quadratic, then
  \[
  \mathcal{Q}^\chi\neq0\quad\iff\quad   \mrm{max}\big\{r_\mrm{alg}^\pm(E/K,\chi)+\delta_p^\pm\big\}=2.
  \]
  If $\chi\colon\Delta_c\to\C^\times$ is not quadratic, then
  \[
  \mathcal{Q}^\chi\neq0\quad\iff\quad   r_\mrm{alg}(E/K,\chi)=2.
  \]
\end{conjecture}
\begin{remark}
    When  $\chi\colon\Delta_c\to\C^\times$ is quadratic, it is easy to see that
    \[
    \mrm{max}\big\{r_\mrm{alg}^\pm(E/K,\chi)+\delta_p^\pm\big\}=2\quad\implies\quad r_\mrm{alg}(E/K,\chi)=2.
    \]
    Thus, the non-vanishing of mock plectic invariants for quadratic characters conjecturally encodes more information than in the  non-quadratic case (see also \cite[Prop. 3.1 $\&$ Theorem 5.1]{polyquadraticPlectic}).
\end{remark}

\begin{remark}
A weaker version of Conjecture \ref{Conj1} can be formulated using coranks of $p$-primary Selmer groups instead of algebraic ranks of elliptic curves. 
The two versions are equivalent if Shafarevich--Tate groups of elliptic curves over number fields are finite.
The main contribution of this article consists in proving one implication of the weaker conjecture for the trivial character (see Theorem \ref{mainTHM}).
\end{remark}

\noindent We conclude the section by describing a conjectural relation between the mock plectic invariant $ \mathcal{Q}^\chi$ and global points in $E(H_c)^\chi$.
Denote by
\[
\mathcal{K}\colon E(H_c)\too \mrm{H}^1_f(H_c,V_p(E))
\]
the Kummer map.
Then, as $p$ is totally split in $H_c/K$, we can choose an embedding $H_c\into K_p$ and consider the associated localization map 
\[
\mrm{loc}_p\colon \mrm{H}^1_f(H_c,V_p(E)) \too \mrm{H}^1_f(K_p,V_p(E)).
\]
We define the $(-a_p)$-isotypic part of the local Kummer map by
\[
\mathcal{K}^{-a_p}_p:= (1-a_p(E)\cdot\sigma_p)\circ\mrm{loc}_p\circ\cal{K}. 
\]
\begin{conjecture}\label{Conj2}
 If $L(E/K,\chi,1)= 0$, then
  $\mathcal{Q}^\chi$ is in the image of the regulator
\[
	\exterior{2}E(H_c)^\chi\too \mrm{H}^1_f\big(H_c, V_p(E)\big)^\chi\otimes_{\Q_p}\mrm{H}_f^1(K_p,V_p(E))^{-a_p},
 \]
 \[
P\wedge Q\mapsto \mathcal{K}(P)\otimes \mathcal{K}_p^{-a_p}(Q)- \mathcal{K}(Q)\otimes \mathcal{K}_p^{-a_p}(P).
\]
\end{conjecture}

\section{Iwasawa theory}
In the second part of this paper we establish one inclusion of Perrin-Riou's Heegner point main conjecture for elliptic curves with multiplicative reduction at $p$ using bipartite Euler systems. As a consequence, we obtain an upper bound on the $p$-Selmer rank of $E_{/K}$ from the non-vanishing of the mock plectic invariant $\cal{Q}_K$.

\subsubsection{Selmer groups.}
We write $\mrm{Sel}_{p^\infty}(E^{\pm}/L)\subseteq \mrm{H}^1(L,E^{\pm}[p^\infty])$ for the $p$-primary Selmer group of $E^{\pm}$ over a subfield $L\subseteq \overline{\Q}$ and set
\[
r_p(E^{\pm}/L):=\mrm{corank}_{\Z_p}\hspace{0.5mm} \mrm{Sel}_{p^\infty}(E^{\pm}/L).
\]
Let $\Delta$ denote the torsion subgroup of $\Gal(H_{p^\infty}/K)$.
By definition, the fixed field
\[
K_\infty:=(H_{p^\infty})^\Delta
\] 
is the anticyclotomic $\Z_p$-extension of $K$.
For non-negative integers $n$ we set $K_n:=H_{p^n}\cap K_\infty$, and consider the Iwasawa algebra $\Lambda:=\Z_p[\![\mathcal{G}]\!]$ of $\mathcal{G}:=\Gal(K_\infty/K)$.
Since $E[p^\infty](K_\infty)$ is finite and $\mathcal{G}$ is topologically generated by one element, it follow from the inflation-restriction exact sequence that the kernel of the restriction map $\mrm{H}^1(K,E[p^\infty])\to \mrm{H}^1(K_\infty,E[p^\infty])$ is finite.
Hence, the commuting diagram with exact rows
\[\xymatrix{
0\ar[r]& \mrm{Sel}_{p^\infty}(E/K)\ar[r]\ar[d]&\mrm{H}^1(K,E[p^\infty])\ar[d]^-{\mrm{res}}\\
0\ar[r]& \mrm{Sel}_{p^\infty}(E/K_\infty)^{\mathcal{G}}\ar[r]&\mrm{H}^1(K_\infty,E[p^\infty])^{\mathcal{G}}\\
}\]
allows us to deduce the following inequality
\begin{equation}\label{bound1}
r_p(E/K) \le \mrm{corank}_{\Z_p}\hspace{0.5mm}\mrm{Sel}_{p^\infty}(E/K_\infty)^{\mathcal{G}}.
\end{equation}
Similarly to \cite[Section 2.2]{KolyvaginES}, we consider the $\Lambda[G_K]$-module
\[
\mbf{T}:=\varprojlim_n \Big(\mrm{Ind}_{K_n/K}\hspace{1mm}T_p(E)\Big)
\]
and its Pontryagin dual
\[
\quad\mbf{W}:=\varinjlim_n\Big(\mrm{Ind}_{K_n/K}\hspace{1mm}E[p^\infty]\Big).
\]
The module $\mbf{T}$ is naturally isomorphic to the tensor product $T_p(E)\otimes_{\Z_p}\Lambda$ with a diagonal $G_K$-action,
while $\mbf{W}$ is the induction from $G_K$ to $G_{K_\infty}$ of $E[p^\infty]$.
In particular, Shapiro's lemma implies that
\[
\mrm{H}^{1}(K,\mbf{W})=\mrm{H}^{1}\left(K_\infty,E[p^\infty]\right).
\]
By our assumptions, $E_{/K}$ has split multiplicative reduction at every $\cal{O}_K$-prime $v\mid pN^-$.
Hence, for every $k\ge1$ the $p^k$-torsion subgroup $E[p^k]$, seen as a $G_{K_v}$-module, fits in the exact sequence 
\begin{equation}\label{SESord}\xymatrix{
0\ar[r]& \mu_{p^k}\ar[r]& E[p^k]\ar[r]& \Z/p^k\Z\ar[r]& 0.
}\end{equation}
We denote by $\mrm{Fil}_v T_p(E)\cong \Z_p(1)$ and $\mrm{Fil}_v E[p^\infty]\cong\mu_{p^\infty}$ the $G_{K_v}$-submodules of $T_p(E)$ and $E[p^\infty]$ respectively determined by the short exact sequence \eqref{SESord}, and consider the $G_{K_v}$-modules
\[
\mrm{Fil}_v(\mbf{T}):=\varprojlim_n \Big(\mrm{Ind}_{K_n/K}\hspace{1mm}\mrm{Fil}_vT_p(E)\Big),\quad\mbox{and}\quad\mrm{Fil}_v(\mbf{W}):=\varinjlim_n \Big(\mrm{Ind}_{K_n/K}\hspace{1mm} \mrm{Fil}_vE[p^\infty]\Big).
\]
The Selmer structure $\mathcal{F}_\Lambda$ on $\mbf{M}=\mbf{T},\mbf{W}$ is defined by taking the unramified local condition at primes of $K$ not dividing $pN^-$, and the ordinary local condition at all $v\mid pN^-$, i.e., the image of
\[
\mrm{H}^1(K_v,\mrm{Fil}_v(\mbf{M}))\to \mrm{H}^1(K_v,\mbf{M}).
\]
If $\mrm{Sel}_{\mathcal{F}_\Lambda}(K,\mbf{M})\subseteq \mrm{H}^1(K,\mbf{M})$ denotes the associated Selmer group,  we set
\begin{equation*}
\scr{S}:=\mrm{Sel}_{\mathcal{F}_\Lambda}(K,\textbf{T})\quad\mbox{and}\quad \cal{X}:=\Hom_{\Z_p}\big(\mrm{Sel}_{\mathcal{F}_\Lambda}(K,\textbf{W}),\Q_p/\Z_p\big).
\end{equation*}
\noindent Interestingly, it is possible to bound the $p$-Selmer rank of $E_{/K}$ in terms of the coinvariants of $\cal{X}$:
\begin{lemma}\label{Selmerbound}
We have
\[
r_p(E/K)\le\mrm{rank}_{\Z_p}\hspace{0.5mm} \mathcal{X}_\mathcal{G}.
\]
\end{lemma}
\begin{proof}
By \eqref{bound1} it is enough to bound the corank of $\mrm{Sel}_{p^\infty}(E/K_\infty)^{\mathcal{G}}$.
The $p$-primary Selmer group $\mrm{Sel}_{p^\infty}(E/K_\infty)$ is a $\mathcal{G}$-submodule of $\mrm{Sel}_{\mathcal{F}_\Lambda}(K,\mbf{W})$ by \cite[Propositions 4.1 $\&$ 4.3]{CoatesGreenberg}.
Hence, 
\[
\mrm{corank}_{\Z_p}\hspace{1mm}\mrm{Sel}_{p^\infty}(E/K_\infty)^\mathcal{G}\le \mrm{corank}_{\Z_p} \hspace{0.5mm}\mrm{Sel}_{\mathcal{F}_\Lambda}(K,\mbf{W})^\mathcal{G} \le \mrm{rank}_{\Z_p}\hspace{0.5mm}\mathcal{X}_\mathcal{G}.
\]
\end{proof}

\noindent The $\Z_p$-rank of $\cal{X}_\cal{G}$ can be studied by means of Perrin-Riou's Heegner point main conjecture.

\subsection{Heegner point main conjecture}
A rational prime $\ell\nmid N$ inert in $K$ is $k$-admissible if $\ell^2\not\equiv1 \pmod{p}$, and if there is a split short exact sequence of $\Gal(K_\ell^\mrm{un}/K_\ell)$-modules
\begin{equation}\label{admissibleses}
\xymatrix{0\ar[r]&\mu_{p^k}\ar[r]& E[p^k]\ar[r]& \Z/p^k\Z\ar[r]& 0.
}\end{equation}
Therefore, for a $k$-admissible prime $\ell$ there exists a unique $\epsilon_\ell\in\{\pm1\}$ such that 
\begin{equation}\label{LRsign}
a_\ell(E)\equiv\epsilon_\ell\cdot(\ell+1)\pmod{p^k}.
\end{equation}
Set $\mbf{T}_k:=\mbf{T}/p^k\mbf{T}$. For every $k$-admissible prime $\ell$ we write $\mrm{H}^1_\mrm{unr}(K_\ell,\mbf{T}_k)\subseteq \mrm{H}^1(K_\ell,\mbf{T}_k)$ for the submodule of unramified classes and $\mrm{H}^1_\mrm{ord}(K_\ell,\mbf{T}_k)\subseteq \mrm{H}^1(K_\ell,\mbf{T}_k)$ for the ordinary submodule, that is, the image of the homomorphism $\mrm{H}^1(K_\ell,\mrm{Fil}_v(\mathbf{T}_k))\to \mrm{H}^1(K_\ell,\mbf{T}_k)$ induced by \eqref{admissibleses}.
 The cohomology groups 
\[
\mrm{H}^1_\mrm{unr}(K_\ell,\mbf{T}_k)\quad \mbox{and}\quad \mrm{H}^1_\mrm{ord}(K_\ell,\mbf{T}_k)
\]
are free of rank one over $\Lambda/p^k\Lambda$ (\cite{BipartiteES}, Lemma 3.12).

\smallskip
\noindent The set of square-free products of $k$-admissible primes is denoted by $\mathcal{N}_k$ and for any $m\in\mathcal{N}_k$ we can consider the Selmer group
\[
\mrm{Sel}_{m}(K,\mbf{T}_k)\subseteq\mrm{H}^1(K,\mbf{T}_k)
\]
of classes which are unramified outside $mpN^-$ and ordinary at every prime divisor of $mpN^-$.
We say that $m\in\mathcal{N}_k$ is \emph{definite} (resp.~\emph{indefinite}) if it has an odd (resp.~even) number of prime factors.
Let $\mathcal{N}_k^\mrm{def}$ and $\mathcal{N}_k^\mrm{ind}$ denote the subsets of $\mathcal{N}_k$ of definite and indefinite elements. 
We now recall the definition of bipartite Euler systems \`a la Howard:
\begin{definition}
A bipartite Euler system for $\mbf{T}$ consists of two collections 
\[
\big\{\kappa_m\in\mrm{Sel}_{m}(K,\textbf{T}_k)\ \big\lvert\ m\in\mathcal{N}_k^\mrm{ind}\big\},\qquad \big\{\lambda_m\in\Lambda/p^k\Lambda\ \big\lvert\ m\in\mathcal{N}_k^\mrm{def}\big\}
\]
for every $k>0$ compatible with respect to the inclusions $\mathcal{N}_{k+1}\subseteq\mathcal{N}_k$, the natural reduction maps $E[p^{k+1}]\xrightarrow{p} E[p^k]$, $\Lambda/p^{k+1}\Lambda\to \Lambda/p^k\Lambda$, and satisfying the \emph{explicit reciprocity laws}:
\begin{itemize}
    \item For any $m\ell\in\mathcal{N}_k^\mrm{ind}$ there is an isomorphism of $\Lambda$-modules $\mrm{H}^1_\mrm{ord}(K_\ell,\mbf{T}_k)\cong \Lambda/p^k\Lambda$ such that
    \[
    \mrm{loc}_\ell(\kappa_{m\ell})=\lambda_m.
    \]
     \item For any $m\ell\in\mathcal{N}_k^\mrm{def}$ there is an isomorphism of $\Lambda$-modules $\mrm{H}^1_\mrm{unr}(K_\ell,\mbf{T}_k)\cong \Lambda/p^k\Lambda$ such that
    \[
    \mrm{loc}_\ell(\kappa_{m})=\lambda_{m\ell}.
    \]
\end{itemize}
\end{definition}

\noindent Since the empty product lies in $\mathcal{N}_k^\mrm{ind}$ for every $k\ge 1$, a bipartite Euler system also gives a distinguished class 
\[
\kappa^\infty:=\varprojlim_k\kappa_1\hspace{1mm}\in\hspace{1mm} \mathscr{S}.
\]

\begin{theorem}\label{IMC}
A bipartite Euler system for $\mathbf{T}$ such that $\kappa^\infty\in\mathscr{S}$ is non-trivial implies that
\[
\mathscr{S}\cong \Lambda\quad\&\quad\mrm{rank}_\Lambda\hspace{0.5mm} \mathcal{X}=1.
\]
Moreover, for any height one prime $\mathcal{P}$ of $\Lambda$, the following inequality holds
\[
\mrm{ord}_\mathcal{P}\big(\mrm{char}(\mathcal{X}_{\Lambda\mbox{-}\mrm{tors}})\big)\le2\cdot\mrm{ord}_\mathcal{P}\big(\kappa^\infty,\mathscr{S}\big),
\]
 where $\mrm{ord}_\mathcal{P}\big(\kappa^\infty,\mathscr{S}\big):=\mrm{sup}\big\{n\ \lvert\ \kappa^\infty\in\mathcal{P}^n\mathscr{S}\big\}$ and $\mathcal{X}_{\Lambda\mbox{-}\mrm{tors}}$ denotes the $\Lambda$-torsion submodule of $\mathcal{X}$.
\end{theorem}
\begin{proof}
This essentially follows from \cite[Theorem 3.2.3]{BipartiteES}.
We only have to prove that $\mathscr{S}$ is free of rank one. Indeed, if $\scr{S}\cong\Lambda$, it is easy to see that
 \[
\mrm{ord}_\mathcal{P}\big(\mrm{char}(\scr{S}/\Lambda\kappa^\infty)\big)=\mrm{ord}_\mathcal{P}\big(\kappa^\infty,\scr{S}\big).
 \] 
As $\kappa^\infty$ is non-zero and $\scr{S}$ is a finitely generated torsion-free $\Lambda$-module by \cite[Lemma 3.2.2]{BipartiteES}, the image of $\kappa^\infty$ is non-trivial in $\scr{S}/\mathcal{P}\scr{S}$ for all but finitely many height one primes $\mathcal{P}$.
Since there are infinitely many Eisenstein polynomials, we can suppose without loss of generality that $\kappa^\infty$ is non-trivial in $\scr{S}/\mathcal{P}\scr{S}$ for some $\mathcal{P}$ generated by an Eisenstein polynomial.
Then, $\Lambda/\mathcal{P}$ is a DVR and $\scr{S}/\mathcal{P}\scr{S}$ is a free rank one $\Lambda/\mathcal{P}$-module by Propositions 3.3.1 and 3.3.3 of \cite{BipartiteES}.
Thus, Nakayama's lemma implies the existence of a surjection $\phi\colon\Lambda\onto\scr{S}$.
We deduce that $\phi$ is an isomorphism because $\scr{S}$ is $\Lambda$-torsion-free and non-zero.
\end{proof}

\subsection{Main theorem}
The remainder of this article is devoted to the construction of a bipartite Euler system for $\mathbf{T}$  (Corollary \ref{corBES}) whose distinguished class $\kappa^\infty$ computes the mock plectic invariant $\cal{Q}_K$ (Lemma \ref{QandKappa}). For the convenience of the reader, we present right away our main theorem and postpone momentarily the most technical parts of the argument.
\begin{theorem}\label{mainTHM}
Suppose $L(E/K,1)=0$ and that $\mathcal{Q}_K\neq 0$, then
    \[
    r_{p}(E/K)= 2.
    \]
   Furthermore, the maximum of $\big\{r_{p}(E^\pm/\Q)+\delta_p^\pm\big\}$ is equal to $2$.
\end{theorem}
\begin{proof}
 Corollary \ref{corBES} and Lemma \ref{QandKappa} together with the assumption $\mathcal{Q}_K\neq0$ imply that there exists a bipartite Euler system for $\mathbf{T}$ whose distinguished class $\kappa^\infty\in\scr{S}$ is non-trivial and such that $\mrm{ord}_I\big(\kappa^\infty,\scr{S}\big)\le1$ for $I\le\Lambda$ the augmentation ideal.
Thus, Theorem \ref{IMC} gives the inequality
\begin{align*}
\mrm{rank}_{\Z_p}\hspace{0.5mm} \mathcal{X}_\mathcal{G}&=1+\mrm{ord}_I\big(\mrm{char}(\mathcal{X}_{\Lambda-\mrm{tors}})\big)\\
&\le1+2\cdot \mrm{ord}_I\big(\kappa^\infty,\scr{S}\big)\\
&\le3,
\end{align*}
which, together with Lemma \ref{Selmerbound}, implies that the $p$-Selmer rank of $E_{/K}$ satisfies $r_p(E/K)\le3$. At the same time, the vanishing of the $L$-value $L(E/K,1)$, Proposition \ref{SelmerProp}, and the non-vanishing of the mock plectic invariant $\mathcal{Q}_K$, ensure that the $p$-Selmer rank $r_p(E/K)$ is positive. Hence,  $r_p(E/K)=2$ thanks to the parity conjecture (for instance \cite[Corollary 12.2.10]{SelmerComplexes}). 

\smallskip
\noindent For the last claim,
note that we have already proved the equality
\[
r_p(E^+/\Q)+r_p(E^-/\Q)=2,
\]
 and that the parity conjecture also tells us that
\[
\varepsilon(E/\Q)=(-1)^{r_p(E^+/\Q)}=(-1)^{r_p(E^-/\Q)}.
\]
Therefore, Corollary \ref{eigenspacePOS} provides the final ingredient to complete the proof, that is, $r_p(E^\circ/\Q)\ge1$ for $\circ$ equal to the sign of $-a_p(E)\cdot\varepsilon(E/\Q)$.
\end{proof}

\begin{remark}
    The converse to Theorem \ref{mainTHM} might be within reach, but its proof will require some new ideas. The main obstacle seems to be proving the full Heegner point main conjecture in our setting. Indeed, Burungale--Castella--Kim's approach relies on having anticyclotomic $p$-adic $L$-functions whose values at the trivial character are non-zero modulo $p$ (see the proof of \cite[Proposition 3.7]{BCK}). However, those $p$-adic $L$-functions acquire an exceptional zero at the trivial character when $p$ is a prime of multiplicative reduction (\cite{CentralCH}, Theorem A).
\end{remark}

\subsection{Level raising}\label{ExpLR}
  For  $m\in\mathcal{N}^\mrm{ind}_k$ let $B_m/\Q$ be the \emph{indefinite} quaternion algebra of discriminant $mN^-$ and choose an isomorphism $B_m\otimes_\Q\R\cong\mrm{M}_2(\R)$. The Shimura curve associated to an Eichler order $R_m\subset B_m$ of level $pN^+$ has a canonical model $Y_m$ defined over $\Q$ and a smooth compactification denoted by $X_m$.
Moreover, the complex manifold $Y_m(\C)$ is biholomorphic to the quotient $\Gamma_m\backslash \mathcal{H}$
where $\Gamma_m:=(R_m)^\times_1$ is the subgroup of elements of $R_m$ of reduced norm $1$.
We write $J_m$ for the Jacobian variety of $X_m$ and $\bb{T}_m$ for the subalgebra of $\End(J_m)$ generated by Hecke correspondences.

\smallskip
\noindent Whenever $m\in\mathcal{N}^\mrm{def}_k$ we consider the Shimura set
\[
S_{m}:=B_{m}^\times\backslash\widehat{B}_{m}^\times/\widehat{R}_{m}^\times
\]
associated to the \emph{definite} quaternion algebra $B_{m}/\Q$ of discriminant $m N^-$ and an  Eichler order $R_{m}\subset B_{m}$ of level $pN^+$.
For every ring $\mathcal{R}$ the free $\mathcal{R}$-module $\mathcal{R}[S_{m}]$ on the set $S_{m}$ can be naturally identified with the space of weight $2$ modular forms of level $pN^+$ and coefficients in $\mathcal{R}$ for the definite quaternion algebra $B_{m}/\Q$.
We denote by $\bb{T}_{m}$ the algebra generated by Hecke operators acting on $\Z[S_{m}]$.

\begin{definition}
Let $m$ be an element of $\mathcal{N}_k$.
		A ring homomorphism $\phi_m\colon \bb{T}_m\to \Z/p^k\Z$ satisfying
    \begin{equation}\label{charCOND}\begin{split}
    \phi_m(T_q)&\equiv a_q(E)\pmod{p^k}\qquad\forall\ q\nmid Nm,\\
    \phi_m(U_q)&\equiv a_q(E)\pmod{p^k}\qquad\forall\ q\mid N,\\
    \phi_m(U_\ell)&\equiv \epsilon_\ell\hspace{7.2mm}\pmod{p^k}\qquad\forall\ \ell\mid m,
   \end{split} \end{equation}
   for $\epsilon_\ell\in\{\pm1\}$ the sign attached to the $k$-admissible prime $\ell\mid m$ as described in \eqref{LRsign},
	is called a \emph{mod $p^k$ level raising} of the system of Hecke eigenvalues attached to the elliptic curve $E_{/\Q}$.
\end{definition}
\noindent A mod $p^k$ level raising $\phi_m$ is clearly unique, if it exists.
We will show that, when $m\in \mathcal{N}^\mrm{ind}_k$ is indefinite, it also induces an isomorphism of $G_\Q$-representations 
\begin{equation}\label{GaloisISO}
T_p(J_m)/\mathcal{I}_{\phi_m}\cong E[p^k]
\end{equation}
 where $\mathcal{I}_{\phi_m}=\ker\phi_m$ and $T_p(J_m)$ denotes the $p$-adic Tate module of $J_m$.

\subsubsection{Existence of mod $p^k$ level raising.}
We sketch a proof of the existence of mod $p^k$ level raisings following the inductive construction in \cite[Sections 5 $\&$ 9]{BD}.
To this end, suppose that $m=\ell_1\ell_2 m_\circ\in \mathcal{N}^\mrm{ind}_k$ is indefinite and that we have constructed $\phi_{m_\circ}\colon \bb{T}_{m_\circ}\to \Z/p^k\Z$ such that  
\[
T_p(J_{m_\circ})/\mathcal{I}_{\phi_{m_\circ}}\cong E[p^k].
\]
When $m_\circ=1$, $\phi_{m_\circ}$ is simply the mod $p^k$ reduction of the system of Hecke eigenvalues attached to the elliptic curve $E_{/\Q}$. 
One begins by constructing a mod $p^k$ level raising surjective eigenfunction $g_{\ell_1m_\circ}\colon S_{\ell_1m_\circ}\to \Z/p^k\Z$ as in the proof of \cite[Theorem 9.3]{BD} (see the proof of Theorem \ref{2ERL} below for more details).
The idea is to consider the local Kummer map for $J_{m_\circ}$ at the prime $\ell_1$ and to exploit the canonical identification between the set of supersingular points $X_{m_\circ}(\bb{F}_{\ell_1^2})^\mrm{ss}$ and the Shimura set $S_{m_\circ\ell_1}$ established by Ribet in \cite[Theorem 3.4]{BimodulesRibet}.
The eigenfunction $g_{\ell_1m_\circ}$ determines a level raising homomorphism $\phi_{\ell_1m_\circ}\colon \bb{T}_{\ell_1m_\circ}\to\Z/p^k\Z$.
To obtain the level raising homomorphism $\phi_{m}\colon\bb{T}_{m}\to\Z/p^k\Z$,
it then suffices to invoke \cite[Theorem 5.15]{BD} together with  Lemma \ref{completionCHAR} below.
Moreover, \cite[Lemma 5.16 $\&$ Theorem 5.17]{BD} apply without change to our setting and provide the isomorphism \eqref{GaloisISO}.

\smallskip
\noindent Let $X_{m_\circ}(\ell_1)$ denote the proper Shimura curve of level $pN^+\ell_1$ attached to the indefinite quaternion algebra $B_{m_\circ}$.
The special fibre at $\ell_1$ of the N\'eron model of its Jacobian is an extension of an abelian variety by a torus.
We denote by $\scr{X}_{\ell_1}(X_{m_\circ}(\ell_1))$ the character group of this torus.

\begin{lemma}\label{completionCHAR}
    The completion of the character group $\scr{X}_{\ell_1}(X_{m_\circ}(\ell_1))$ at the unique maximal ideal $\mathfrak{m}\subseteq \bb{T}_{\ell_1m_\circ}$ associated to $\phi_{\ell_1m_\circ}$ is free of rank one over the completed Hecke algebra $(\bb{T}_{\ell_1m_\circ})_{\mathfrak{m}}$.  
\end{lemma}
\begin{proof}
 As in the proof of \cite[Theorem 6.2]{PW11}, it suffices to show that the $\mathfrak{m}$-torsion of the Jacobian of $X_{m_\circ}(\ell_1)$ is two dimensional over $\bb{F}_p$. Since the residual representation $\overline{\varrho}$ is $p$-distinguished, the dimension over $\bb{F}_p$ of $J_0(N\ell_1m_\circ)[\mathfrak{m}]$ is two by \cite[Theorem 2.1(ii)]{W}.
Moreover, the maximal ideal $\mathfrak{m}$ is controllable (in the terminology of \cite{Helm}) at every prime divisor of $m_\circ$ and $N^-$ either because of the definition of admissible primes or because of our assumption that $\overline{\varrho}$ is ramified at every prime $\ell\mid N^-$ with $\ell^2\equiv1\pmod{p}$. Hence, \cite[Corollary 8.11 $\&$ Remark 8.12]{Helm} allow one to deduce the claim for the Jacobian of $X_{m_\circ}(\ell_1)$.
 
 \noindent We remark that Skinner--Zhang noted in the proof of \cite[Lemma 5.5]{SkinnerZhang} that \cite[Corollary 8.11 $\&$ Remark 8.12]{Helm} do not cover the case where the residual representation is not finite at $p$. The gap is traced back to \cite[Lemma 7.1]{Helm} which incorrectly quotes \cite[Theorem A]{DiamondTaylor}. Fortunately for us, the level raising lemma \cite[Lemma 5.4]{SkinnerZhang} completes Helm's argument. 
\end{proof}

\begin{remark}
 As an historical note that might be helpful in navigating the literature, we observe that the proof of \cite[Proposition 9.2]{BD} -- a key input for \cite[Theorem 9.3]{BD} -- is incorrect as written because it invokes the version of Ihara's lemma proved by Diamond--Taylor (cf.~\cite{DiamondTaylor}, Theorem 2) which requires the level of the Shimura curve to be coprime to the characteristic of the coefficients.
A solution of this issue was found by Chida--Hsieh (see \cite{ChidaHsieh}, Corollary 5.4) under the additional assumption $a_p(E)^2\not\equiv1\pmod{p}$, and is now completely resolved thanks to the general version of Ihara's lemma established by Manning--Shotton (\cite{Manning-Shotton}, Theorem 1.1). 
 \end{remark}

\subsubsection{Mod $p^k$ multiplicity one.} 
Let $m$ be an element of $\mathcal{N}^\mrm{def}_k$ and $\mathcal{R}$ a ring such that $6\in \mathcal{R}^\times$.
Let us consider the perfect paring
\begin{equation}\label{pairing}
\langle\ ,\hspace{1mm}\rangle\colon \mathcal{R}[S_{m}]\times \mathcal{R}[S_{m}]\longrightarrow \mathcal{R},\qquad \big\langle[a],[b]\big\rangle=\begin{cases}
    0&\mbox{if}\ [a]\neq[b],\\
    w_{[a]}^{-1}&\mbox{if}\ [a]=[b],
\end{cases}
\end{equation}
on $\mathcal{R}[S_{m}]$ defined in \cite[Equation 2.14]{ChidaHsieh}.
The constants $w_{[a]}\ge1$ are integers not divisible by primes larger than $3$ (see \cite{HilbertIwasawa}, Lemma 3.3).
Moreover, the Hecke operators in $\bb{T}_{m}$ are self-adjoint with respect to the pairing by \cite[Lemma 3.5]{HilbertIwasawa}.
Set 
\[
(\Z/p^k\Z)[S_{m}]^{\phi_{m}}:=\Big\{h\in (\Z/p^k\Z)[S_{m\ell}]\ \Big\lvert\ T\cdot h=\phi_{m}(T)\cdot h\quad \forall\ T\in\bb{T}_{m\ell}\Big\}.
\]
The following result establishes the uniqueness, up to multiplication by a unit in $(\Z/p^k\Z)^\times$, of a surjective mod $p^k$ level raising eigenfunction $g_{m}\colon S_{m}\to\Z/p^k\Z$. 
\begin{corollary}\label{multiplicityONE}
For any $m\in\mathcal{N}_k^\mrm{def}$ we have 
\[
(\Z/p^k\Z)[S_{m}]^{\phi_{m}}\cong\Z/p^k\Z.
\]
\end{corollary}
\begin{proof}
 For any $\ell\mid m$ one can show that the completion of the character group $\scr{X}_{\ell}(X_{m/\ell}(\ell))$ at $\mathfrak{m}$ is isomorphic, as an Hecke module, to the completion of $\Z[S_{m}]$ at $\mathfrak{m}$. The proof uses \cite[Proposition 5.3]{BD} and the fact that the maximal ideal $\mathfrak{m}\subseteq\bb{T}_{m}$ associated to $\phi_{m}$ is not Eisenstein.
	Therefore, Lemma \ref{completionCHAR} implies that $(\Z/p^k\Z)[S_{m}]/\mathcal{I}_{\phi_{m}}$ is abstractly isomorphic to $\Z/p^k\Z$. The corollary follows because the perfect pairing \eqref{pairing} induces an isomorphism 
    \[
    (\Z/p^k\Z)[S_{m}]^{\phi_{m}}\xlongrightarrow{\sim} \Hom_{\Z/p^k\Z}\Big((\Z/p^k\Z)[S_{m}]/\mathcal{I}_{\phi_{m}},\Z/p^k\Z\Big).
   \]
\end{proof}

\subsection{Bipartite Euler system}\label{compFAM}
The goal of the last part of this paper is to construct a bipartite Euler system for $\mathbf{T}$ following \cite{BD}, \cite{PW11}, \cite{ChidaHsieh}, and \cite{BCK}.
We emphasize the structural parallels between the construction of the classes $\kappa_m$'s and that of the elements $\lambda_m$'s to clarify why they satisfy the explicit reciprocity laws.

\subsubsection{Compatible families of CM points.}
For any $m\in\cal{N}_k^\mrm{ind}$, the Shimura curve $Y_{m/\Q}$ coarsely represents a moduli problem of abelian surfaces with an action of a maximal order $R_\mrm{max}\supset R_m$ and level $pN^+$-structure, that is, a subgroup scheme of order $(pN^+)^2$ which is stable and cyclic for the action of  $R_\mrm{max}$ (see \cite{BD}, Section 5.1).
    Let $c$ be a positive integer prime to $p$ and $\psi\colon K\hookrightarrow B_m$ an embedding of conductor $cp$, that is, satisfying
    \[
\psi(K)\cap R_m=\psi(\mathcal{O}_{cp}).
    \]
The image of the local ring  $\mathcal{O}_{K_p}$ under $\psi$ is contained in a maximal $\Z_p$-order of $B_m\otimes_\Q\Q_p$.
Fix an isomorphism $\iota_p\colon B_m\otimes_\Q\Q_p\xrightarrow{\sim}\mrm{M}_2(\Q_p)$ mapping that maximal order to $\mrm{M}_2(\Z_p)$, then the action of $K_{p}^\times$ on the Bruhat--Tits tree $\scr{T}$ induced by $\iota_p\circ\psi$ fixes the standard vertex $v_\circ\in\scr{V}$.

\smallskip
\noindent The point $[z_{\psi}]\in\mrm{CM}_K(X_m)$ associated to the $\Gamma_m$-conjugacy class of $\psi$  belongs to $X_m(H_{cp})\subseteq X_m(\C)$. Under the moduli interpretation $[z_{\psi}]$  corresponds to a triple  $(A, \iota, C,C_p)_{/H_{cp}}$ where $C$ and $C_p$ are stable and cyclic subgroup schemes of $A$ (with respect to  the quaternionic action)  and whose respective orders are $(N^+)^2$ and $p^2$. 
     Let $\scr{T}_{(A,\iota,C)}$ denote the  tree whose vertices are triples $(A',\iota',C')$ which are $p$-isogenous to $(A,\iota,C)$, and whose oriented edges corresponds to isogenies of degree $p^2$. Note that the CM point $[z_{\psi}]$ determines an oriented edge  in $\scr{T}_{(A,\iota,C)}$. Every vertex of $\scr{T}_{(A,\iota,C)}$ is defined over the union of ring class fields $H_{cp^\infty}\subset\C$ and the Galois group $\Gal(H_{cp^\infty}/H_{cp})\cong\scr{U}_1\subseteq K_{p}^\times/\Q_p^\times$ acts naturally on the tree. The choice of a trivialization $\mrm{H}_1(A(\C),\Z_p)\cong R_\mrm{max}\otimes\Z_p$ determines a graph isomorphism 
    \[
    j_{(A,\iota,C)}\colon \scr{T}_{(A,\iota,C)}\xlongrightarrow{\sim} \scr{T}
    \]
    which maps a vertex $(A',\iota',C')$ to the vertex corresponding to the maximal order
    \[\mrm{H}_1(A'(\C),\Z_p)\ \subset\  \mrm{H}_1(A(\C),\Q_p)\hspace{0.7mm}\cong\hspace{0.7mm} B_m\otimes_\Q\Q_p.
    \]
 The theory of complex multiplication implies that $j_{(A,\iota,C)}$ intertwines the $\scr{U}_1$-action  on $\scr{T}_{(A,\iota,C)}$ via Galois automorphisms with the $\scr{U}_1$-action  on $\scr{T}$ induced by $\iota_p\circ\psi$. 
    
    \noindent We deduce that there is a bijection between sequences of consecutive edges $\{e_n^\psi\}_{n\ge1}$ originating from $v_\psi$, as defined in Section \ref{actionontree}, and collections of CM points $\{x^{\psi}_n\in X_m(H_{cp^n})\}_{n\ge1}$ 
satisfying 
 \begin{equation}\label{traceCMpts}
\mrm{Trace}_{H_{cp^{n+1}}/H_{cp^n}}(x^{\psi}_{n+1})=U_p(x^{\psi}_n).
 \end{equation}
Moreover, $\Gal(H_{cp^\infty}/H_{cp})\cong\scr{U}_1$ acts transitively on the collections of those sequences with $e_1^\psi$ corresponding to $[z_{\psi}]$ (see also \cite[Section 6]{BD} and \cite[Section 3.2]{DarmonFornea}).

\subsubsection{The classes $\kappa_m$.} For an embedding  $\psi\colon K\hookrightarrow B_m$ of conductor $p$ choose a collection of CM points $\{x^{\psi}_n\in X(H_{p^n})\}_{n\ge1}$ satisfying \eqref{traceCMpts}. 
The choice of an auxiliary prime $\ell_\circ$ such that $a_{\ell_\circ}(E)\not\equiv\ell_\circ+1\pmod{p}$ allows us to consider the map 
\[
\delta_m\colon X_m(H_{p^n})\to J_m(H_{p^n})\otimes\Z_p,\qquad x\mapsto (T_{\ell_\circ}-\ell_\circ-1)[x]\otimes(a_{\ell_\circ}(E)-\ell_\circ-1)^{-1}.
\]
Denoting by $\mathcal{K}\colon J_m(H_{p^n})\otimes_\Z\Z_p\to\mrm{H}^1\big(H_{p^n},T_p(J_m)\big)$
 the Kummer map, one sees that the classes defined by 
\[
\kappa^\diamond_m(p^n):=a_p(E)^{-n}\sum_{\sigma\in G(H_{p^n}/K_n)}\mathcal{K}(\delta_m(x^{\psi}_n)^\sigma)\pmod{\mathcal{I}_{\phi_m}}
\]
don't depend on the choice of the auxiliary prime $\ell_\circ$, are compatible under corestriction, and hence define an element in the projective limit 
\[
\kappa^\diamond_m:=\big\{\kappa^\diamond_m(p^n)\big\}_n\hspace{0.1mm}\in\hspace{0.6mm}\varprojlim_n \mrm{H}^1(K_n,E[p^k])\cong \mrm{H}^1(K,\mbf{T}_k).
\]
\begin{proposition}\label{localPROP}
Let $q\in p\Z_p$ denote the Tate period of $E_{/\Q_p}$. For every $m\in\mathcal{N}_k^\mrm{ind}$ we have 
\[
\mrm{ord}_p(q)\cdot\kappa^\diamond_m\in\mrm{Sel}_{m}(K,\mbf{T}_k)\subseteq\mrm{H}^1(K,\mbf{T}_k).
\]
\end{proposition}
\begin{proof}
The localization of the class $\kappa^\diamond_m$ at $v\nmid mN$ is unramified because the Jacobian $J_m$ has good reduction at such a prime. It is unramified at any $v\mid N^+$ thanks to \cite[Lemma 1.4(1)]{ChidaHsieh} since any such prime is required to split in $K$.  Moreover, the fact that the localization at some $v\mid mN^-$ belongs to the ordinary subspace is a direct consequence of Tate's uniformization for $J_m$ as can be seen for example in \cite[Section 1.7.3]{NekCanadian}.

   \smallskip
   \noindent We are left to show that for all $n$ large enough the class  $\mrm{ord}_p(q)\cdot\kappa^\diamond_m(p^n)\in \mrm{H}^1(K_n,E[p^k])$ is ordinary at $p$. Since the elliptic curve $E_{/K_p}$ has split multiplicative reduction, \cite[Proposition 2.7.8(3)]{NekCanadian} implies that the ordinary local condition for $E[p^k]$ at $p$ coincides with the Bloch--Kato condition obtained through propagation from  $V_p(E)$ as in \cite[Section 2.7.1]{NekCanadian}. We claim that \cite[Proposition 2.7.12(2)]{NekCanadian} allows us to conclude the proof. To see this, note that the Jacobian $J_{m/\Q_p}$ achieves split semi-abelian reduction over an \emph{unramified polyquadratic} extension $F'/\Q_p$ (see also \cite[Proposition 2.7.3]{NekCanadian}). Hence, the constant $C_{0,p}$ -- appearing in \cite[Definition 2.7.9]{NekCanadian} -- is zero (see also \cite[A.1.2 $\&$ A.1.7]{NekCanadian}). Moreover, the constant $C_{6,p}$ -- defined in \cite[Proposition-Definition 2.7.10]{NekCanadian} -- equals
   \[
C_{6,p}=\mrm{ord}_p\big(\mrm{ord}_p(q)\big).
   \]
Then, \cite[Proposition 2.7.12(2)] {NekCanadian}  gives the claim since $p>2$ and the degree $[K_{n,p}F':K_{n,p}]$  is always a power of $2$ in our setting.
\end{proof}

\begin{definition}
   For any $m\in\cal{N}_k^\mrm{ind}$ we set 
    \[
\kappa_m:=\mrm{ord}_p(q)\cdot\kappa^\diamond_m\hspace{1mm}\in\hspace{1mm}\mrm{Sel}_{m}(K,\mbf{T}_k).
    \]
\end{definition}

\subsubsection{Connection to mock plectic invariants.}
Let $I\subseteq \Lambda$ be the augmentation ideal.
The map $g\mapsto [g]-1$ induces an isomorphism $\mathcal{G}\xrightarrow{\sim} I/I^2$ of topological groups, and the composition
\[\xymatrix{
K_{p,1}^{\times}\ar[r]^-{\sim}& K_p^{\times}/\Q_p^{\times}\ar[r]^-{\mrm{rec}_K}& \mathcal{G}\ar[r]^-{\sim}& I/I^2
}\]
allows us to view the mock plectic invariant $\mathcal{Q}_K$ as an element of $\mrm{H}^1(K, V_p(E)\otimes_{\Z_p} I/I^2)$. It turns out that we can interpret $\mathcal{Q}_K$ as the first derivative of the distinguished class 
\[
\kappa^\infty=\varprojlim_{k} \kappa_1\hspace{1mm}\in\hspace{1mm} \scr{S}\subset \mrm{H}^1(K, T_p(E)\otimes_{\Z_p} \Lambda)
\]
of our bipartite Euler system:
\begin{lemma}\label{QandKappa}
The class $\kappa^\infty$ belongs to $\mrm{H}^1(K, V_p(E)\otimes_{\Z_p} I)$.
Moreover,  the equality
\begin{equation}
\mrm{ord}_p(q)\cdot\mathcal{Q}_K\equiv\kappa^\infty\pmod{I^2}
\end{equation}
holds in $\mrm{H}^1(K, V_p(E)\otimes_{\Z_p} I/I^2)$.
\end{lemma}
\begin{proof}
 Up to replace the elliptic curve $E$ by a $\Q$-isogenous one, we can assume that the modular parametrization  $\varphi\colon X\to E$ identifies $E$ with the quotient of the Jacobian of $X$ determined by the system of Hecke eigenvalues attached to $E$. Then, the choice of an embedding $\psi\colon K\hookrightarrow B$ of conductor $p$ and that of a sequence of consecutive edges $\{e_n^\psi\}_{n\ge1}$ determine a collection of CM points $\{x^{\psi}_n\in X(H_{p^n})\}_{n\ge1}$  such that 
\[
\kappa^\diamond_1(p^n)=\mbox{Kummer image of  }\mrm{Trace}_{H_{p^n}/K_n}\big({^{\psi}\hspace{-.5mm}P_{n}}\big)
\]
in $\mrm{H}^1(K_n,T_p(E))$ for every $n\ge1$, where  ${^{\psi}\hspace{-.5mm}P_{n}}$ is the point defined in equation \eqref{thepoints}.
 Now, the first claim follows from Lemma \ref{compatiblePOINTS}, while the second is deduced from the description of $\mathcal{Q}_K$ as a limit of Kolyvagin derivatives given in equation \eqref{KolyDerivative}.
\end{proof}

\subsubsection{The elements $\lambda_{m}$.}
Let $m\in\mathcal{N}_k^\mrm{def}$ and $\mathcal{R}$ be any ring. By strong approximation, an automorphic form $g\in \mathcal{R}[S_{m}]$ of level $pN^+$ on the definite quaternion algebra $B_{m}$ can be interpreted as $(R_{m}[1/p])^\times$-invariant function $g\colon \scr{E}\to \mathcal{R}$ on oriented edges of the Bruhat--Tits tree $\scr{T}$ for $\mrm{PGL}_2(\Q_p)$ such that the $U_p$-operator is described by the formula
\[
U_p(g)(e)=\sum_{s(e')=t(e)}g(e')\qquad\forall\ e\in\scr{E}.
\]
Now, any embedding $\psi\colon K\hookrightarrow B_{m}$ satisfying 
\begin{equation}\label{conditionawayfromp}
\psi(K)\cap R_{m}[1/p]=\psi(\mathcal{O}_K[1/p])
\end{equation}
induces, via the reciprocity map of class field theory and strong approximation, an action $\bs{\psi}\colon \Gal(H_{p^\infty}/K)\to \mrm{Aut}(\scr{T})$ of $\Gal(H_{p^\infty}/K)$ on the Bruhat--Tits tree $\scr{T}$. From any sequence $\{e_n^\psi\}_{n\ge1}$ of consecutive edges originating from $v_\psi$, we can then define the  $\Z[\Gal(H_{p^n}/K)]$-valued automorphic forms on $B_{m}$ of level $pN^+$ by setting
\[
h_n^\psi:=\sum_{\sigma\in G(H_{p^n}/K)}[\sigma]\cdot \big[\bs{\psi}_\sigma( e_n^\psi)\big].
\]
 If we denote by $\pi\colon \Z[\Gal(H_{p^{n+1}}/K)]\to \Z[\Gal(H_{p^n}/K)]$ the natural projections, we find that
\[
\pi\big(h_{n+1}^\psi\big)=U_p\big(h_{n}^\psi\big).
\]
Let $\phi_{m}\colon\bb{T}_{m}\to\Z/p^k\Z$ be the mod $p^k$ level raising eigensystem constructed in Section \ref{ExpLR}. Then, the  elements 
\[
\Big\{a_p(E)^{-n}\cdot h_{n}^\psi\ \in\ \Z[S_{m}]/\mathcal{I}_{\phi_{m}}\otimes_\Z\Z[\Gal(H_{p^n}/K)]\Big\}_{n\ge1}
\]
are compatible and, through the projections $\Gal(H_{p^n}/K)\to \Gal(K_n/K)$, define an element in the projective limit
\begin{equation}\label{generatingSeries}
 \tilde{\lambda}^\diamond_{m}\ \in\ \Z[S_{m}]/\mathcal{I}_{\phi_{m}}\otimes_\Z\Lambda.
\end{equation}
\begin{remark}
    We dropped the reference to the embedding $\psi$ from the symbol $\tilde{\lambda}^\diamond_{m}$ because the element depends mildly on the choices involved in its construction. A different sequence of consecutive edges originating from $v_\psi$ multiplies the automorphic forms $h_n^\psi$'s by a fixed group-like element in $\Gal(H_{p^\infty}/H)$, while a different embedding satisfying \eqref{conditionawayfromp} multiplies the automorphic forms by a group-like element in $\Gal(H_{p^\infty}/K)$ (see \cite{CDuniformization}, Section 2).
\end{remark}
\noindent A mod $p^k$ level raising eigenfunction $g_{m}$ as constructed in Section \ref{ExpLR} gives
\[
\lambda^\diamond_{m}:=\langle  \tilde{\lambda}^\diamond_{m},g_{m}\rangle\hspace{1mm}\in\hspace{1mm}\Lambda/p^k\Lambda
\]
which depends on the choice of $g_{m}$ only up to multiplication by a unit in $\Z/p^k\Z$ (Corollary \ref{multiplicityONE}).

\begin{definition}
   For any $m\in\cal{N}_k^\mrm{def}$ we set 
    \[
\lambda_{m}:=\mrm{ord}_p(q)\cdot\lambda^\diamond_{m}\hspace{1mm}\in\hspace{1mm}\Lambda/p^k\Lambda.
    \]
\end{definition}

\subsubsection{Explicit reciprocity laws.}
The reciprocity laws are proved following the arguments in \cite[Sections 8 $\&$ 9]{BD} (see also \cite[Theorems 5.1 $\&$ 5.5]{ChidaHsieh}). For the convenience of the reader, we briefly recall the main ideas.
\begin{theorem}
For every $m\in\mathcal{N}_k^\mrm{ind}$ and any $k$-admissible prime $\ell\mid m$, there is an isomorphism of $\Lambda$-modules $\mrm{H}^1_\mrm{ord}(K_\ell,\mbf{T}_k)\cong \Lambda/p^k\Lambda$ such that
\[
\mrm{loc}_\ell(\kappa_m)=\lambda_{m/\ell}.
\]
\end{theorem}
\begin{proof}[Sketch of proof]
    Since $\ell$ is an admissible prime, the subspace of local classes at $\ell$ which are ordinary maps isomorphically to the singular quotient. Thus, Shapiro's lemma and \cite[Section 1.7.3]{NekCanadian} provide an isomorphism 
    \[
\Phi_{m,\ell}/\mathcal{I}_{\phi_m}\otimes_\Z\Lambda\cong \mrm{H}^1_\mrm{ord}(K_\ell, \textbf{T}_k),
    \]
    where $\Phi_{m,\ell}$ denotes the group of connected components of the N\'eron model of $J_{m/K_\ell}$ (see also \cite[Theorem 4.5]{ChidaHsieh}). The quotient $\Phi_{m,\ell}/\mathcal{I}_{\phi_m}$ can be explicitly identified with $\Z[S_{m/\ell}]/\mathcal{I}_{\phi_{m/\ell}}$ by \cite[Theorem 4.3]{ChidaHsieh} in such a way that the image of the class $\mrm{loc}_\ell(\kappa_m)$ in $\Z[S_{m/\ell}]/\mathcal{I}_{\phi_{m/\ell}}\otimes_\Z\Lambda$ equals the inverse limit of the automorphic forms defined in \eqref{generatingSeries} for some  $\psi\colon K\hookrightarrow B_{m/\ell}$ satisfying \eqref{conditionawayfromp}.
    Finally, the choice of a mod $p^k$ level raising eigenfunction $g_{m/\ell}$ provides the isomorphism $\langle-,g_{m/\ell}\rangle\colon \Z[S_{m/\ell}]/\mathcal{I}_{\phi_{m/\ell}}\xrightarrow{\sim}\Z/p^k\Z$ and hence an isomorphism $\mrm{H}^1_\mrm{ord}(K_\ell, \textbf{T}_k)\cong \Lambda/p^k\Lambda$ which satisfies the claim of the theorem.
\end{proof}

\begin{theorem}\label{2ERL}
For every $m\in\mathcal{N}_k^\mrm{ind}$ and any $k$-admissible prime $\ell\nmid m$, there is an isomorphism of $\Lambda$-modules $\mrm{H}^1_\mrm{unr}(K_\ell,\mbf{T}_k)\cong \Lambda/p^k\Lambda$ such that
\[
\mrm{loc}_\ell(\kappa_m)=\lambda_{m\ell}.
\]
\end{theorem}
\begin{proof}[Sketch of proof]
We begin by explaining the construction of a mod $p^k$ level raising eigenfunction $g_{m\ell}\colon S_{m\ell}\to\Z/p^k\Z$ as mentioned in Section \ref{ExpLR}.
    Since $\ell\nmid pm$, the local Kummer map $J_m(K_\ell)/\mathcal{I}_{\phi_m}\to\mrm{H}^1(K_\ell,E[p^k])$ factors as the composition of surjective homomorphisms
    \[\xymatrix{
    J_m(K_\ell)/\mathcal{I}_{\phi_m}\ar@{->>}[r]& J_m(\bb{F}_{\ell^2})/\mathcal{I}_{\phi_m}\ar@{->>}[r]^-{\mathcal{K}_\ell}& \mrm{H}_\mrm{unr}^1(K_\ell,E[p^k])
    }\] 
    where $\mrm{H}_\mrm{unr}^1(K_\ell,E[p^k])\cong\Z/p^k\Z$ because $\ell$ is $k$-admissible.
    Recall that Ribet gave a canonical identification between the set of supersingular points $X_m(\bb{F}_{\ell^2})^\mrm{ss}$ and the Shimura set $S_{m\ell}$. Then, Ihara's lemma can be used to prove that the natural map
    \[\xymatrix{
    \Z[S_{m\ell}]/\mathcal{I}_{\phi_{m\ell}}\ar@{->>}[r]&J_m(\bb{F}_{\ell^2})/\mathcal{I}_{\phi_m}
    }\]
    is surjective (see \cite{BD}, Theorem 9.2). We deduce the existence of a unique surjective mod $p^k$ level raising eigenfunction $g_{m\ell}\colon S_{m\ell}\to\Z/p^k\Z$ such that 
    \begin{equation}\label{comparingKummer}
    \mathcal{K}_\ell(-)=\langle-,g_{m\ell}\rangle.
    \end{equation}
    Finally, Shapiro's lemma and the equality in \eqref{comparingKummer} reduce the claim of the theorem to showing that the natural image of $\mrm{loc}_\ell(\kappa_m)$ in $\Z[S_{m\ell}]/\mathcal{I}_{m\ell}\otimes_\Z\Lambda$ is given by the inverse limit of the automorphic forms defined in \eqref{generatingSeries} for an embedding  $\psi\colon K\hookrightarrow B_{m\ell}$ satisfying \eqref{conditionawayfromp}.
\end{proof}

\begin{corollary}\label{corBES} 
The collections
\[
\big\{\kappa_m\in\mrm{Sel}_{m}(K,\mathbf{T}_k)\ \big\lvert\ m\in\mathcal{N}_k^\mrm{ind}\big\},\qquad \big\{\lambda_m\in\Lambda/p^k\Lambda\ \big\lvert\ m\in\mathcal{N}_k^\mrm{def}\big\}
\]
 for every $k>0$ form a bipartite Euler system for $\mbf{T}$.
 \end{corollary}

\bibliography{Plectic}
\bibliographystyle{alpha}

\end{document}